\documentclass{amsart}

\usepackage[a4paper,hmargin=3.5cm,vmargin=4cm]{geometry}
\usepackage{amsfonts,amssymb,amscd,amstext}
\usepackage{graphicx}
\usepackage[dvips]{epsfig}
\usepackage{mathpazo}
\usepackage{enumerate}

\newcommand{\esca}[1]{\langle {#1} \rangle}
\newcommand{\doble}[1]{\ll \hspace{-0.1cm} {#1} \hspace{-0.1cm} \gg}
\newcommand{\curvo}[1]{\prec \hspace{-0.1cm} {#1} \hspace{-0.1cm} \succ}

\def\dist{{\rm dist}}
\def\Ccal{\mathcal{C}}
\def\Dcal{\mathcal{D}}

\def\Ncal{\mathcal{N}}
\def\Mcal{\mathcal{M}}
\def\Wgot{\mathfrak{W}}
\def\Rcal{\mathcal{R}}
\def\Hcal{\mathcal{H}}
\def\c{\mathbb{C}}
\def\d{\mathbb{D}}
\def\R{\mathbb{R}}\def\r{\mathbb{R}}
\def\Z{\mathbb{Z}}\def\z{\mathbb{Z}}
\def\N{\mathbb{N}}\def\n{\mathbb{N}}
\def\s{\mathbb{S}}
\def\h{\mathbb{H}}
\def\k{\mathbb{K}}

\newtheorem{theorem}{Theorem}[section]
\newtheorem{proposition}[theorem]{Proposition}
\newtheorem{claim}[theorem]{Claim}
\newtheorem{lemma}[theorem]{Lemma}
\newtheorem{corollary}[theorem]{Corollary}
\newtheorem{remark}[theorem]{Remark}
\newtheorem{definition}[theorem]{Definition}

\theoremstyle{definition}

\numberwithin{equation}{section}

\usepackage{xcolor}

\begin{document}

\title[Null Curves in $\c^3$ and Calabi-Yau Conjectures]
{Null Curves in $\c^3$ and Calabi-Yau Conjectures}

\author[A.~Alarc\'{o}n]{Antonio Alarc\'{o}n}
\address{Departamento de Geometr\'{\i}a y Topolog\'{\i}a \\
Universidad de Granada \\ E-18071 Granada \\ Spain}
\email{alarcon@ugr.es}

\author[F.J.~L\'{o}pez]{Francisco J. L\'{o}pez}
\address{Departamento de Geometr\'{\i}a y Topolog\'{\i}a \\
Universidad de Granada \\ E-18071 Granada \\ Spain}
\email{fjlopez@ugr.es}


\thanks{Research partially
supported by MCYT-FEDER research project MTM2007-61775 and Junta
de Andaluc\'{i}a Grant P09-FQM-5088}


\begin{abstract} For any  open orientable surface $M$ and convex domain $\Omega\subset \c^3,$ there exist a  Riemann surface $N$ homeomorphic to $M$ and a complete proper  null curve  $F:N\to\Omega.$ This result follows from a general existence theorem with many applications. Among them, the followings:
\begin{itemize}
\item For any convex domain $\Omega$ in $\c^2$ there exist   a Riemann surface $N$ homeomorphic to $M$ and a complete proper holomorphic immersion $F:N\to\Omega.$  Furthermore, if $D \subset \r^2$ is a convex domain and $\Omega$ is the solid right cylinder $\{x \in \c^2 \,|\, \mbox{Re}(x) \in D\},$ then $F$ can be chosen so that ${\rm Re}(F):N\to D$ is proper. 
\item There exist a Riemann surface $N$ homeomorphic to $M$ and a complete bounded holomorphic null immersion $F:N \to {\rm SL}(2,\c).$
\item There exists a complete bounded CMC-1 immersion $X:M \to \h^3.$
\item For any convex domain $\Omega \subset \r^3$ there exists a complete proper minimal immersion $(X_j)_{j=1,2,3}:M \to \Omega$  with vanishing flux. Furthermore, if $D \subset \r^2$ is a convex domain and $\Omega=\{(x_j)_{j=1,2,3} \in \r^3 \,|\, (x_1,x_2) \in D\},$ then $X$ can be chosen so that $(X_1,X_2):M\to D$ is proper. 
\end{itemize} 
Any of the above surfaces can be chosen with hyperbolic conformal structure.
\end{abstract}

\maketitle

\thispagestyle{empty}


\section{Introduction}\label{sec:intro}

Calabi \cite{calabi} asked whether there exist complete minimal surfaces in a bounded domain, or more generally,  with bounded projection into a straight line. This problem is known in the literature as the {\em Calabi-Yau problem} for immersed minimal surfaces in $\r^3.$ Variations of this problem which have been considered deal with immersed surfaces which
are also proper in a domain of $\r^3.$ The first answer to  Calabi's question was given by Jorge and Xavier \cite{j-x}, who exhibited complete non-flat minimal discs in a slab of $\r^3.$ Later, Yau \cite{ya1,ya2} revisited these conjectures and opened new lines for research. It was Nadirashvili  \cite{nadi} who developed a powerful technique for constructing  complete bounded minimal surfaces in  $\r^3.$ His examples are minimal discs with never vanishing Gaussian curvature, providing  counterexamples to classical Hadamard's conjecture for negatively curved surfaces as well. Both Jorge-Xavier's and Nadirashvili's constructions make use of the classical Runge approximation theory for holomorphic functions on planar labyrinths of compact sets. These ideas have been a fountain of insight that has strongly influenced the global theory of minimal surfaces along the last decade. L\'{o}pez, Mart\'{i}n and Morales \cite{l-m-m} added handles to Nadirashvili's surfaces, and some years later  Alarc\'{o}n, Ferrer, Mart\'{i}n and Morales \cite{m-m1,m-m2,a-f-m} constructed proper complete minimal surfaces in smooth domains of $\r^3,$ under some restrictions on the topology of the surfaces and the geometry of the domains.  Recently, Ferrer, Mart\'{i}n and Meeks \cite{f-m-m}  have given  a complete solution to the proper Calabi-Yau problem for minimal surfaces of arbitrary topology in both convex and smooth bounded domains of $\r^3,$ even with disjoint limit sets for distinct ends.

The embedded Calabi-Yau  problem for minimal surfaces has radically  different nature. Colding and Minicozzi \cite{c-m} have proved that any  
complete embedded minimal surface in $\r^3$ with finite topology is proper in $\r^3,$ and  Meeks, P\'{e}rez and Ros \cite{m-p-r} have extended this result to the family of minimal surfaces with finite genus and countably many ends.

Calabi-Yau and  Hadamard's conjectures are closely related and make sense for a wide range of  surfaces and ambient manifolds. Given a Riemannian manifold $\mathfrak{M},$ the (immersed) Calabi-Yau problem in $\mathfrak{M}$ deals with the existence of complete proper submanifolds in domains of $\mathfrak{M},$ with a given constrain on its geometry (minimal, CMC, non-positively curved,...).  This paper is devoted to the corresponding Calabi-Yau problem for null curves in $\c^3,$ or being more precise, to the existence of complete proper {\em null curves} in convex domains of  $\c^3.$ 

Given an open Riemann surface $N,$ a map $F=(F_j)_{j=1,2,3}:N \to \c^3$ is said to be a null curve in $\c^3$ if $F$ is a holomorphic immersion and  $\sum_{j=1}^3 (dF_j)^2=0.$ The Riemannian metric on $N$ induced by the Euclidean metric in $\c^3$ is given by $ds^2_F:=\sum_{j=1}^3 |dF_j|^2.$

Calabi-Yau problem for null curves in $\c^3$ is interesting by itself, but also because it provides a key for solving other  Calabi-Yau problems, namely, for null curves in ${\rm SL}(2,\c),$ holomorphic immersed curves in $\c^2,$  CMC-1 surfaces (or Bryant surfaces) in $\h^3$ and minimal surfaces in $\r^3.$

Jones \cite{jones} constructed a complete bounded  holomorphic immersion of the unit disc $\d$ in $\c^2,$ a complete bounded  holomorphic embedding of $\d$ in $\c^3,$  and a  complete proper holomorphic immersion of $\d$ in the unit ball of $\c^4.$ As a consequence, he produced complete bounded  null curves in $\c^n$ for any $n \geq 4.$ On the other hand, Bourgain \cite{bou} showed that there are no complete bounded null curves in $\c^2.$ However, the Calabi-Yau problem for null curves in $\c^3$ has remained open for a long time. 
The first approach to this question was made by Mart\'{i}n, Umehara and Yamada \cite{muy1,MUY3}.

Our Main Theorem below (see Theorem \ref{th:main}) represents a wide generalization of  all these results, and as we will see later, has many interesting consequences. For a rigorous statement, the following notations are required.

Throughout this paper we adopt column notation for both vectors and matrices of linear transformations in $\c^3,$ and identify $$\c^3 \ni (z_1,z_2,z_3)^T \equiv ({\rm Re}(z_1),{\rm Im}(z_1), {\rm Re}(z_2),{\rm Im}(z_2),{\rm Re}(z_3),{\rm Im}(z_3))^T \in \r^6,$$ where as usual $(\cdot)^T$ means "transpose".

\begin{definition} \label{def:wide}
If $\rho=\{\rho_i\}_{1\leq i\leq n} \subseteq \{1,\ldots,6\}$ is a strictly increasing sequence, $n \geq 1,$  and $\rho^*=\{\rho^*_i\}_{1\leq i\leq 6-n}$ is the (possibly void) complementary one in   $\{1,\ldots,6\},$ we denote by
$\r^\rho=\{(x_j)_{j=1,\ldots,6} \in \r^6\,|\, x_j=0 \; \forall \, j \in \rho^*\},$ and 
$\Pi_\rho:\r^6 \to \r^\rho$ the corresponding Euclidean orthogonal projection.

The sequence  $\rho$  is said to be  {\em wide} if $n\geq 2$ and $\rho \neq \{2j-1,2j\},$ $j=1,2,3.$  

Given  $\Omega \subset \r^\rho,$ we denote by $\mathcal{C}_\rho(\Omega)$ the cylinder $\{x \in \r^6\,|\, \Pi_\rho(x) \in \Omega\}.$  When $\rho^*=\emptyset$ then $\r^\rho=\r^6,$  $\Pi_\rho={\rm Id}_{\r^6}$ and $\mathcal{C}_\rho(\Omega)=\Omega,$  and we make the conventions $\r^{\rho^*}=\{\vec{0}\}$ and $\Pi_{\rho^*}\equiv \vec{0}.$
\end{definition}

Our main result asserts:

\begin{quote}
{\bf Main Theorem.} {\em 
Let $\rho$ be a wide sequence in $\{1,\ldots,6\},$ and let $\Omega$ be a convex domain in $\r^\rho$ (possibly all $\r^\rho$).

Then for any open orientable surface $M$ there exist a hyperbolic Riemann surface $N$ homeomorphic to $M$ and a complete null curve $F:N\to \mathcal{C}_\rho(\Omega)$ such that   $\Pi_\rho\circ F:N\to \Omega$ is proper.}
\end{quote}

It is classically known that any open hyperbolic Riemann surface $\Mcal$ carries neither proper holomorphic functions $f:\Mcal \to \c$ nor proper harmonic functions $h:\Mcal \to \r.$  
Therefore, Main Theorem does not hold if $\rho$ is not wide and $\Omega=\r^\rho,$ and in this sense is sharp (see Remark \ref{falla1}).

Our construction method is different to the ones in \cite{jones, muy1}. The main tools in this paper come from approximation theory by meromorphic functions, but in this case without using labyrinths of compact sets. A key point is to use a Mergelyan's type approximation result which provides an enormous capability for modeling null curves in $\c^3$ (see Lemma \ref{lem:runge} and \cite{al}). Our arguments rely only on the geometry of $\c^3,$ and the involved approximation results for null curves are of {\em extrinsic} nature. Roughly speaking, the null curve $F$ in the theorem is obtained by deforming recursively a sequence of compact null curves in $\Ccal_\rho(\Omega).$ Unlike previous methods, during the deformation we have direct control over the immersion itself instead of over its  Weierstrass data. Furthermore, completeness and properness can be achieved at the same time in the process and checked extrinsically as well. 

Different choices of sequence $\rho$ and convex domain $\Omega \subset \r^\rho$ generate a list of suggestive corollaries. The most straightforward
one is the existence of complete bounded null curves in $\c^3.$

\begin{quote} {\bf Corollary I [Calabi-Yau problem in $\c^3$].} {\em For any open orientable surface $M$ and any convex domain $\Omega$ in $\c^3$ (possibly $\Omega=\c^3$), there exist a  Riemann surface $N$ homeomorphic to $M$ and a  proper complete null curve $F:N\to\Omega.$ } 
\end{quote}
In particular, this answers affirmatively the Problem 1 in \cite{muy1}: ``{\em Are there complete null curves properly immersed in the unit ball of $\c^3$?}". 
A partial result in the line of Corollary I can also be found in \cite{a-f-l}. 

Denote by $\lfloor \,,\,\rfloor$ the Hermitian inner product in ${\rm SL}(2,\c)$ given by $\lfloor A,B \rfloor={\rm trace}(A \cdot \bar{B}^T),$ $A,B\in{\rm SL}(2,\c).$ The operator $\bar{\cdot}$ means complex conjugation. A map $Z:N\to {\rm SL}(2,\c)$ is said to be a {\em null curve} in ${\rm SL}(2,\c)$ if $Z$ is a holomorphic immersion  and $\det (dZ)=0.$ The Riemannian metric  on $N$ induced by $\lfloor \,,\,\rfloor$ is given by $ds_Z^2 =\lfloor dZ,dZ\rfloor.$ The following correspondence is a biholomorphism preserving null curves (see \cite{muy1}):
\begin{equation} \label{eq:corres}
\mathcal{T}:\c^3-\{z_3=0\}\to \{(a_{ij})\in{\rm SL}(2,\c)\,|\, a_{11}\neq 0\},\;\;
\mathcal{T}((z_j)_{j=1,2,3})=\frac1{z_3}\left(\begin{matrix} 1 & z_1+\imath z_2\\ z_1-\imath z_2 & z_1^2+z_2^2+z_3^2\end{matrix}\right),
\end{equation}
$\imath=\sqrt{-1}.$  
The transformation $\mathcal{T}$ provides complete bounded null curves in ${\rm SL}(2,\c)$ when applied to complete bounded null curves in $\c^3-\{|z_3|>1\}$ \cite{muy1}. Unfortunately, proper examples in ${\rm SL}(2,\c)$ can not be constructed by this procedure.  Taking into account Corollary I, we  get that:

\begin{quote}{\bf Corollary II [Calabi-Yau problem in ${\rm SL}(2,\c)$].} {\em For any open orientable surface $M,$ there exist a  Riemann surface $N$ homeomorphic to $M$ and a complete bounded null curve $F:N \to {\rm SL}(2,\c).$}
\end{quote}

Let $\h^3=\{(x_0,x_1,x_2,x_3) \in \r^4\,|\, x_1^2+x_2^2+x_3^2+1=x_0^2,\,x_0>0\}$ be the hyperboloid model of the 3-dimensional hyperbolic space. We call $\langle,\,\rangle_0$ as the hyperbolic metric in $\h^3$ induced by the 4-dimensional Lorentz-Minkowski space $\mathbb{L}^4$ of signature $(-,+,+,+).$   Up to the canonical identification $$(x_0,x_1,x_2,x_3) \equiv  \left( \begin{array}{ccc}
x_0+x_3 & x_1 +\imath x_2 \\
x_1-\imath x_2 & x_0-x_3
\end{array} \right),$$ $\h^3=\{A \cdot \bar{A}^T\,|\,A\in{\rm SL}(2,\c)\}.$ With this language, Bryant's projection $\mathcal{B}:{\rm SL}(2,\c) \to \h^3,\; \mathcal{B}(A)=A\cdot \bar{A}^T,$ maps null curves in ${\rm SL}(2,\c)$  into conformal immersions of  mean curvature $H=1$ in $\h^3$. Furthermore, if $Z:N \to {\rm SL}(2,\c)$ is a null curve then the pull back $(Z\cdot \bar{Z}^T)^* \langle,\,\rangle_0$ coincides with $\frac1{2}ds_Z^2$ (see \cite{bry, u-y} for a good setting). 

The family  of complete CMC-1 surfaces in $\h^3$ with finite topology  is very vast, see \cite{r-u-y,p-p} for a good reference. For the arbitrary topology case,
there is no general existence result available known to the authors. Regarding  Calabi-Yau questions, and as pointed out in \cite{muy1}, applying Bryant's projection $\mathcal{B}$ to the complete bounded null curves of Corollary II we get that:
\begin{quote}{\bf Corollary III [Calabi-Yau problem in $\h^3$].} {\em For any open orientable surface $M,$ there exists a complete bounded CMC-1 immersion $X:M \to \h^3.$}
\end{quote}

Mart\'{i}n, Umehara and Yamada \cite{muy2}  extended Jones' existence result \cite{jones} to complete bounded complex submanifolds with arbitrary finite genus and finitely many ends in $\c^2.$ On the other hand, the existence of proper holomorphic immersions in $\c^2$ with arbitrary topological type is well known \cite{bis,nar,rem,al}.  From Main Theorem  it follows considerably more:

\begin{quote} {\bf Corollary IV [Calabi-Yau problem in $\c^2$].} {\em For any open orientable surface $M$ and any convex domain $\Omega$ in $\c^2$ (possibly $\Omega=\c^2$), there exist   a Riemann surface $N$ homeomorphic to $M$ and a complete proper holomorphic immersion $F:N\to\Omega.$ 

Furthermore, if $D \subset \r^2$ is a convex domain and $\Omega$ is the solid right cylinder $\{x \in \c^2 \,|\, {\rm Re}(x) \in D\},$ then $F$ can be chosen so that ${\rm Re}(F):N\to D$ is proper.} 
\end{quote}

The real part of a null curve in $\c^3$ is a minimal immersion in $\r^3$ with {\em vanishing flux}, that is to say, such that the integral of the conormal vector to the immersion  along any arc-length parameterized closed curve in the surface vanishes.   As a consequence of  Main Theorem,
\begin{quote} 
{\bf Corollary V [Calabi-Yau problem in $\r^3$].} {\em For any open orientable surface $M$ the following assertions hold:
\begin{enumerate}[{\rm (i)}]
\item For any convex domain $\Omega \subset \r^3$ (possibly $\Omega=\r^3$), there exists a complete proper minimal immersion $X:M\to\Omega$ with vanishing flux. 
\item For any convex domain $D \subset \r^2$ (possibly $D=\r^2$), there exists a complete minimal immersion  $X=(X_j)_{j=1,2,3}:M\to \r^3$ with vanishing flux such that $(X_1,X_2)(M)\subset D$ and  $(X_1,X_2):N\to D$ is proper.
\item There exists a bounded complete flux vanishing minimal immersion $X:M \to \r^3$ such that all its associate immersions are bounded.
\end{enumerate}}
\end{quote}
Although certainly Corollary V-(i) is strongly related to Ferrer-Mart\'{i}n-Meeks theorem \cite{f-m-m}, these results do not imply each other. 
Recently, the authors \cite{al} have constructed minimal surfaces with arbitrary conformal structure properly projecting into $\r^2,$ answering a question posed by Schoen and Yau \cite[p. 18]{s-y}.  Corollary V-(ii) shows that the analogous result for convex domains of $\r^2$ holds as well. 

Finally, we remark that all the open Riemann surfaces involved in the above corollaries are of hyperbolic conformal type.


\section{Preliminaries}
We denote by $\|\cdot\|$ and ${\rm dist}(\cdot,\cdot)$ the Euclidean norm and distance in $\k^n,$ where $\k=\r$ or $\c,$ and for any compact topological space $K$ and continuous map $f:K \to \k^n$ we denote by 
\[\|f\|=\max\{\|f(p)\|\,|\, p \in K\}
\]
the maximum norm of $f$ on $X.$

Given an $n$-dimensional topological manifold $M,$ we denote by $\partial M$ the $(n-1)$-dimensional topological manifold determined by its boundary points. For any  $A \subset M,$ $A^\circ$ and $\overline{A}$ will denote the interior and the closure of $A$ in $M,$ respectively.  Open connected subsets of $M-\partial M$ will be called {\em domains}, and those proper topological subspaces of $M$ being $n$-dimensional manifolds with boundary are said to be  {\em regions}. If $M$ is a topological surface,  $M$ is said to be {\em open} if it is non-compact and $\partial M =\emptyset.$ 

\subsection{Riemann surfaces}

An open Riemann surface is said to be {\em hyperbolic} if it carries non constant negative subharmonic functions.

\begin{remark}
Throughout this paper $\Ncal$ and $\sigma_\Ncal^2$ will denote a fixed but arbitrary open hyperbolic Riemann surface and a conformal Riemannian metric on it, respectively.
\end{remark}

In the following we introduce the necessary notations for a well understanding of the paper. 

A Jordan arc in $\mathcal{ N}$ is said to be analytical if it is contained in an open analytical Jordan arc in $\mathcal{ N}.$

Classically, a compact region $A\subset\Ncal$ is said to be Runge if $\Ncal-A$ has no bounded (i.e., relatively compact in $\Ncal$) components, or equivalently, if the inclusion map $\iota_A: A\hookrightarrow \Ncal$ induces a group monomorphism  $(\iota_A)_*:\Hcal_1(A,\z) \to \Hcal_1(\Ncal,\z),$ where $\Hcal_1(\cdot,\z)$ means first homology group with integer coefficients. For convenience we will extend this notion to a general subset $A\subset\Ncal,$ and say that a   $A$ is {\em Runge} if  $(\iota_A)_*:\Hcal_1(A,\z) \to \Hcal_1(\Ncal,\z)$ is injective. In this case we identify the groups  $\Hcal_1(A,\z)$ and  $(\iota_A)_*(\Hcal_1(A,\z)) \subset \Hcal_1(\Ncal,\z)$  via $(\iota_A)_*$ and consider $\Hcal_1(A,\z) \subset \Hcal_1(\Ncal,\z).$ 

Two Runge subsets $A_1,$ $A_2\subset \Ncal$ are said to be {\em isotopic} if $\Hcal_1(A_1,\z)= \Hcal_1(A_2,\z).$ Two Runge subsets $A_1,$ $A_2 \subset \Ncal$ are said to be {\em homeomorphically isotopic} if there exists a homeomorphism $\sigma: A_1 \to A_2$ such that $\sigma_*={\rm Id}_{\Hcal_1(A_1,\z)},$ where $\sigma_*$ is  the induced group morphism on homology. In this case $\sigma$ is said to be an isotopical homeomorphism. Two Runge domains (or compact regions) with finite topology  in $\Ncal$  are isotopic if and only if they are homeomorphically isotopic.

\begin{definition}[Admissible set]
A compact subset $S\subset\Ncal$ is said to be admissible if and only if:
\begin{itemize}
\item $M_S:=\overline{S^\circ}$ is a finite collection of pairwise disjoint compact regions in $\Ncal$ with   $\mathcal{ C}^0$ boundary,
\item $C_S:=\overline{S-M_S}$ consists of a finite collection of pairwise disjoint analytical Jordan arcs, 
\item any component $\alpha$ of $C_S$  with an endpoint  $P\in M_S$ admits an analytical extension $\beta$ in $\Ncal$ such that the unique component of $\beta-\alpha$ with endpoint $P$ lies in $M_S,$ and
\item $S$ is Runge.
\end{itemize}
\end{definition}

Let $W$ be a Runge domain  of finite topology in $\Ncal$, and let $S$ be an admissible subset in $\Ncal.$ $W$ is said to be a {\em tubular neighborhood} of $S$ if $S \subset W,$ $S$ is isotopic to $W$ and $\chi(W-S)=0,$ where $\chi(\cdot)$ means Euler characteristic. In other words, if $S \subset W$ and $W-S$ consists of a finite collection of pairwise disjoint open annuli.

For any subset $A \subset \Ncal,$ we denote by  
\begin{itemize}
\item ${\mathcal{ F}_0}(A)$ the space of continuous functions $f:A \to{\c}$ which are holomorphic on an open neighborhood  of $A$ in $\Ncal,$ 
\item $\mathcal{ F}_0^*(A)$ the space of continuous functions $f:A \to \c$ being holomorphic on $A^\circ.$ 
\item $\Omega_0(A)$ the space of holomorphic
1-forms on an open neighborhood of $A$ in $\Ncal,$ and 
\item $\Omega_0^*(A)$ the space of complex 1-forms $\theta$ of type $(1,0)$ that are continuous on $A$ and holomorphic on $A^\circ.$ As usual, a 1-form $\theta$ on $A$ is said to be of type $(1,0)$ if for any conformal chart $(U,z)$ in $ \mathcal{ N},$ $\theta|_{U \cap A}=h(z) dz$ for some function $h:U \cap S \to \c.$
\end{itemize}

Let $S$ be an admissible subset of $\Ncal.$ 

A function $f \in \mathcal{F}_0^*(S)$ is said to be {\em smooth} if $f|_{M_S}$
admits a smooth extension $f_0$ to a domain $W$ containing $M_S,$ and for any component $\alpha$ of $C_S$
and any open analytical Jordan arc $\beta$ in $\Ncal$ containing $\alpha,$  $f$ admits a smooth extension $f_\beta$ to $\beta$
satisfying that $f_\beta|_{W \cap \beta}=f_0|_{W \cap \beta}.$ 

A 1-form $\theta\in \Omega_0^*(S)$ is said to be {\em smooth} if,
for any closed conformal disk $(U,z)$ on $\Ncal$ such that $ S\cap U$ is admissible, the function $\theta/dz\in \mathcal{F}_0^*(S)$ is smooth. 

Given a smooth function $f\in \mathcal{F}_0^*(S),$ we set $df \in \Omega_0^*(S)$
as the smooth 1-form given by $df|_{M_S}=d (f|_{M_S})$ and $df|_{\alpha \cap U}=(f \circ \alpha)'(x)dz|_{\alpha \cap U},$
where $(U,z=x+i y)$ is a conformal chart on $\Ncal$ such that $\alpha \cap U=z^{-1}(\R \cap z(U)).$
A smooth 1-form $\theta \in \Omega_0^*(S)$ is said to be {\em exact} if $\theta=df$ for some smooth $f \in \mathcal{F}_0^*(S),$
or equivalently if $\int_\gamma \theta=0$ for all $\gamma \in \mathcal{ H}_1(S,\Z).$

The $\mathcal{C}^1$-norm on $S$ of a smooth $f\in \mathcal{F}_0^*(S)$ is defined by
\[
\|f\|_1=\max\{\|f(p)\|+\|\frac{df}{\sigma_\Ncal}(p)\|\,| \;p \in S\}.
\]

A sequence of smooth functions $\{f_n\}_{n \in \n}\subset \mathcal{F}_0^*(S)$ is said to converge in the {\em $\Ccal^1$-topology} to a smooth function $f\in \mathcal{F}_0^*(S)$ if $\{\|f-f_n\|_1\}_{n \in \n}\to 0.$ If in addition $f_n$ is (the restriction to $S$ of) a holomorphic function on an open neighborhood $W$ of $S$ in $\Ncal$ for all $n,$  we also say that $f$ can be {\em uniformly $\Ccal^1$-approximated} on $S$ by  functions in $\mathcal{F}_0(W).$

Likewise one can define the notions of smoothness, (vectorial) differential,  $\Ccal^1$-norm and uniform $\Ccal^1$-approximation for maps $f:S \to \c^k,$ $k \in \n.$

%


\subsection{Null Curves in $\c^3$ and complex orthogonal transformations} \label{sec:null}
 
Let us start this subsection by introducing some operators which are strongly related to the
geometry of $\c^3$ and null curves. Let $A\subset \c^3.$ We denote by
\begin{itemize}
\item $\doble{\,,\,}:\c^3\times\c^3\to \c,$ $\doble{ u,v }=\bar{u}^T \cdot v,$ the usual Hermitian inner product in $\c^3,$  
\item $\doble{ A}^\bot=\{v\in\c^3\,|\, \doble{ u,v}=0 \, \forall u \in A\},$  
\item $\langle\,,\,\rangle={\rm Re}(\doble{\,,\,}):\c^3\times\c^3\to\r$ the Euclidean scalar product of $\c^3\equiv\r^6,$ 
\item $\langle A\rangle^\bot=\{v\in\c^3\,|\, \langle u,v \rangle=0 \, \forall u \in A\},$
\item $\curvo{\,,\,}:\c^3\times\c^3\to\c$ the complex symmetric bilinear 1-form given by $\curvo{  u, v} = u^T\cdot v,$ and 
\item $\curvo{  u}^\bot=\{v\in\c^3\,|\,\curvo{  u, v}=0\}.$ 
\end{itemize}

Notice that $\curvo{ \overline{u}}^\bot= \doble{ u}^\bot\subset\esca{u}^\bot$ for all $u\in\c^3,$ and the equality holds iff $u=\vec{0}:=(0,0,0)^T.$ 

A vector $u \in \c^3-\{\vec{0}\}$ is said to be {\em null} if $\curvo{ u,u}=0.$ We denote by 
\[
\Theta=\{u \in \c^3-\{\vec{0}\} \,|\, u \;\mbox{is null}\}.
\]

\begin{remark}\label{rem:nulidad}
$\Theta=\{(\frac12 z (1-w^2),\frac{i}{2} z(1+w^2),z w) \,|\, w,\,z \in \c,\, z\neq 0 \}.$ As a consequence, $\Theta$ is a complex conical submanifold of $\c^3$ not contained in a finite union of real (or complex)  hyperplanes of $\c^3.$ 
\end{remark}

A basis $\{u_1,u_2,u_3\}$ of $\c^3$ is said to be {\em $\curvo{\,,\,}$-conjugate} if $\curvo{ u_j,u_k}=\delta_{jk},$ $j,k\in\{1,2,3\}.$ Likewise we define the notion of $\curvo{\,,\,}$-conjugate basis of a complex subspace $U,$ provided that $\curvo{\,,\,}|_{U \times U}$ is a non degenerate complex bilinear form.

We denote by $\mathcal{O}(3,\c)$ the complex orthogonal group $\{A\in\mathcal{M}_3(\c)\,|\, A^T A=I_3\},$ i.e., the group of matrices  whose column vectors determine a $\curvo{ \,,\,}$-conjugate basis of $\c^3.$ As usual, we also denote by $A:\c^3 \to \c^3$ the complex linear transformation induced by $A \in \mathcal{O}(3,\c).$ Observe that
\begin{equation}\label{eq:ATheta}
\curvo{Au,Av}=\curvo{u,v} \quad\text{and}\quad \doble{A u,\overline{A} v}=\doble{u,v}, \quad\forall u,v\in\c^3,\; A\in\mathcal{O}(3,\c).
\end{equation}

Let $M$ be an open Riemann surface. Using the above language, a holomorphic map $F:M\to\c^3$ is a null curve iff $\curvo{ \Phi,\Phi}=0$ and $\doble{\Phi,\Phi}$ never vanishes on $M,$ where $\Phi=dF.$ Conversely, given an  exact holomorphic vectorial 1-form $\Phi$ on $M$ satisfying that $\curvo{ \Phi,\Phi}=0$ and $\doble{\Phi,\Phi}$ never vanishes on $M,$ the map $F:M\to\c^3,$ $F(P)=\int^P \Phi,$ defines a null curve in $\c^3.$ In this case $\Phi$ is said to be the {\em Weierstrass representation} of $F.$

If $F:M\to\c^3$ is a null curve, then the pull back metric $ds_F^2:= F^* \esca{\,,\,}$ on $M$ coincides with  $\esca{dF,dF}=\doble{dF,dF}.$
\begin{definition}
Given two subsets $V_1, V_2\subset M,$ we denote by ${\rm dist}_{(M,F)}(V_1,V_2)$ the intrinsic distance  between $V_1$ and $V_2$ in the Riemannian surface $(M,ds_F^2).$
\end{definition}
\begin{remark}\label{rem:AF}
Let $F:M\to\c^3$ be a null curve and $A=(a_{jk})_{j,k=1,2,3}\in\mathcal{O}(3,\c).$ Then $A\circ F:M\to\c^3$ is also a null curve and $ds_F^2\geq \frac1{ \|A\|^2} ds_{A\circ F}^2,$ where $\|A\|\,= \big(\sum_{j,k}  |a_{jk}|^2  \big)^{1/2}.$
\end{remark}


The following definitions deal with the notion for {\em null curve} on  admissible subsets.

\begin{definition}
Given a proper subset $M\subset \Ncal,$ we denote by ${\sf N}(M)$ the space of maps $X:M \to\c^3$ extending as a null curve to an open neighborhood of $M$ in $\Ncal.$
\end{definition}

\begin{definition}\label{def:gen-null}
Let $S\subset \Ncal$ be an admissible subset.
A smooth map $F\in\mathcal{F}_0^*(S)^3$ is said to be a {\em generalized null curve} in $\c^3$ if it satisfies the following properties:
\begin{itemize}
\item $F|_{M_S}\in {\sf N}(M_S)$ and
\item $\curvo{ dF,dF}=0$ and  $\doble{ dF,dF }$ never vanishes on $S.$
\end{itemize}
\end{definition}

The following Mergelyan's type result for null curves is a key tool in this paper. It will be required to approximate generalized null curves by null curves defined on larger domains.

\begin{lemma}\label{lem:runge}
Let  $W\subset \Ncal$ be a Runge domain of finite topology,  and let $S$ be a connected admissible compact set contained in $W$ and isotopic to $W.$ 
Let $F=(F_j)_{j=1,2,3}:S \to \c^3$ be a generalized null curve.

Then $F$ can be uniformly $\mathcal{C}^1$-approximated on $S$ by a sequence $\{H_n=(H_{j,n})_{j=1,2,3}\}_{n\in\N}$ in ${\sf N}(W).$ In addition, we can choose
$H_{3,n}=F_3$ for all $n\in\N$ provided that $F_3\in \mathcal{F}_0(W)$ and $dF_3$ never vanishes on $C_S.$
\end{lemma}
\begin{proof} Use the Approximation Lemma in \cite{al} for $\Phi=dF$ to get a sequence of exact vectorial 1-forms $\{\Phi_n\}_{n\in\n}\subset \Omega_0(W)^3$ converging to $dF$ uniformly on $S.$ Since  $S$ is isotopic to $W,$ $H_n:=F(P_0)+\int_{P_0}\Phi_n$ is well defined on $W$  for all $n\in\n,$ where $P_0$ is any point in $M_S.$  $\{H_n\}_{n \in \n}$ solves the lemma.
\end{proof}

\subsection{Convex domains}

Throughout this section, $\mathcal{D}$ will denote a regular convex domain of $\r^n,$ $\Dcal \neq \r^n,$  $n\geq 2.$ 

Recall that $\Dcal \cap (p+T_p \, \partial \Dcal)=\emptyset$ for all $p \in \partial \Dcal,$ where $T_p \, \partial \Dcal$ denotes the real tangent space of $\partial \Dcal$ at $p.$  Therefore $\overline{\Dcal}=\cap_{p \in \partial \Dcal} H_p,$ where $H_p$ is the closed half space bounded by $p+T_p \, \partial \Dcal$ and containing  $\Dcal,$ $p \in \partial \Dcal.$ 

Let $\nu_\Dcal: \partial \mathcal{D} \rightarrow \mathbb{S}^{n-1}$ be the
outward pointing unit normal of $\partial \mathcal{D}.$ 
Given $p\in\partial \mathcal{D}$ and $v \in T_p\, \partial \Dcal\cap \,\s^{n-1},$  we denote by  $\kappa_\Dcal (p,v)$ the  normal curvature at $p$ in the direction of $v$   with respect to $-\nu_\Dcal,$ obviously non-negative. In particular, the principal curvatures of $\partial \mathcal{D}$ at $p$ with respect to $-\nu_\Dcal$ are non-negative.  We denote by $\kappa(p)\geq 0$
the maximum of these principal curvatures at $p \in \partial \Dcal,$ and by
\[
\kappa (\mathcal{D}):= \sup \{\kappa(p)\;|\;p\in\partial \mathcal{D}\} \in [0,+\infty].
\]
If $p \in \partial \Dcal,$ $v \in \s^{n-1} \cap T_p \, \partial \Dcal$ and  $\kappa_\Dcal (p,v)>0,$ basic convex geometry gives that  
$\lim_{\lambda \to \infty} {\rm dist}(p+\lambda v,\overline{\Dcal})=+\infty.$ The  domain $\mathcal{D}$ is said to be {\em strictly convex} if $\kappa_\Dcal (p,v)>0$ for all $p \in \partial \Dcal$ and $v \in \s^{n-1} \cap T_p \, \partial \Dcal.$

\begin{definition} \label{def:dt} 
For any $t \in (-\frac1{\kappa(\Dcal)}, +\infty)$ we denote by $\Dcal_t$ the convex domain in $\r^n$ bounded by $\partial \Dcal_t=\{p+ t\cdot \nu_\Dcal(p) \, | \, p \in \partial
\mathcal{D}\}$ and such that  $\Dcal\subset \Dcal_t$ if $t\geq 0,$ and $\Dcal_t\subset \Dcal$ if $t\leq 0.$
\end{definition}
We have made the conventions $-\frac1{\kappa(\Dcal)}=-\infty$ and $-\frac1{\kappa(\Dcal)}=0$ provided that $\kappa(\Dcal)=0$ and $\kappa(\Dcal)=+\infty,$ respectively. We label $\Dcal_{-1/\kappa(\Dcal)}$ as the closed subset $\cap_{t>-1/\kappa(\Dcal)} \Dcal_t.$ Note that $\partial \Dcal_t$ is a regular (convex) hypersurface $\forall t \in (-\frac1{\kappa(\Dcal)}, +\infty).$ 

Set $\pi_\mathcal{D}:\r^n-\mathcal{D}_{-1/\kappa(\Dcal)}\to \partial \mathcal{D}$ as the normal projection given by $\pi_\mathcal{D}(p+t\nu_\Dcal(p))=p,$ and keep denoting by  $\nu_\Dcal$ the {\em extended normal map} $\nu_\Dcal\circ \pi_\Dcal:\r^n-\mathcal{D}_{-1/\kappa(\Dcal)} \to \s^{n-1}.$

\begin{definition}\label{def:escapingvector}
A vector $v \in \r^n-\{\vec{0}\}$ is said to be a {\em escaping vector} in $\Dcal$ if $\overline{\Dcal}$ contains no half lines parallel to $v,$ or equivalently, if $\lim_{\r \ni \lambda \to \infty} \dist(p+\lambda v,\overline{\Dcal})=+\infty$ for all $p \in \r^n.$ In this case 
\[
\liminf_{\r \ni \lambda \to \infty} \frac{1}{|\lambda|}{\rm dist}(p+\lambda v,\overline{\Dcal})>0\quad \text{ for all $p \in \r^n.$}
\]
We denote by $\mathcal{E}_\Dcal$ the set of a escaping vectors in $\Dcal.$
\end{definition}
The set $\mathcal{E}_\Dcal$ is  empty if and only if  $\Dcal$ contains a half space, and otherwise it is the complement in $\r^n$ of a double cone with vertex $\vec{0}$ and convex base.  If $p \in \partial \Dcal$ and $v \in T_p \, \partial \Dcal -\mathcal{E}_\Dcal$ then $\partial \Dcal\cap T_p \, \partial \Dcal$ contains a half line parallel to $v$ and with initial point $p,$ whereas  $v \in \mathcal{E}_\Dcal\cap T_p \, \partial \Dcal$ implies that  $\partial \Dcal\cap \{p+\lambda v\,|\,  \lambda \in \r\}$ is a compact segment containing $p.$ If $\Dcal$ is strictly convex then  $T_p \,\partial \Dcal-\{\vec{0}\}\subset \mathcal{E}_\Dcal,$  $\partial \Dcal\cap (p+T_p\, \partial \Dcal)=\{p\}$ and 
$\lim_{T_p \, \partial \Dcal \ni v \to \infty}  \dist(p+ v,\overline{\Dcal})=+\infty$ for all $p \in \partial \Dcal.$

Assume that $\kappa(\Dcal)<+\infty$ and take $r \in (0,1/\kappa(\Dcal))$ and $p \in \Dcal-\overline{\Dcal_{-r}}.$ Consider $\delta\in (0,r)$ and a neighborhood $U_p$ of $p$  so that $U_p \subset  \Dcal-\overline{\Dcal_{-r+\delta}}$ and ${\rm diam}(U_p)<\delta.$  Then it is straightforward to check that 
\begin{equation} \label{eq:dista} 
\text{$\overline{\Dcal_{-r}} \cap (q_1+T_{\pi_\Dcal (q_2)} \, \partial \Dcal)=\emptyset$ for all $q_1,$ $q_2 \in U_p.$}
\end{equation}


\begin{remark} \label{re:pita}
Let $\Dcal$ and $\hat{\Dcal}$ be two regular convex domains in $\r^n$ with $\kappa(\Dcal)<+\infty$ and  $\overline{\Dcal}\subset \hat{\Dcal}.$ Consider $p\in\partial \Dcal,$ and take $r\in [0,1/\kappa(\Dcal))$ and $q\in\partial\Dcal_{-r}$ with $\pi_\Dcal(q)=p.$ By convexity and basic trigonometry, one has that ${\rm dist}(q,(q+T_q \partial \mathcal{D}_{-r})\cap\partial \hat{\Dcal})\geq\sqrt{d_r^2+2\frac{d_r}{\kappa(\mathcal{D}_{-r})}},$ where $d_r={\rm dist}(\Dcal_{-r},\partial \hat{\Dcal}).$ 
Since $T_q \partial \mathcal{D}_{-r}=T_p \partial \mathcal{D},$ $d_r\geq d_0+r$ and  $\kappa(\mathcal{D}_{-r})=\frac{\kappa(\mathcal{D})}{1-r\kappa(\mathcal{D})},$ we infer that
$${\rm dist}(q,(q+T_p \partial \mathcal{D})\cap\partial \hat{\Dcal})\geq  \sqrt{d_0^2+2\frac{d_0}{\kappa(\mathcal{D})}}.$$
\end{remark}

Given two compact subsets $C,D\subset\r^n$,
the Hausdorff distance between $C$ and $D$ is defined by
\[
\delta^H(C,D)=\max\left\{ \sup_{x\in C} \inf_{y\in D} \|x-y\|\;,\;
\sup_{y\in D} \inf_{x\in C} \|x-y\|\right\}.
\]
A sequence $\{K_n\}_{n \in \n}$ of (possibly unbounded) closed subsets of $\r^n$ is said to be convergent in the Hausdorff topology  to a closed subset $K_0$ of $\r^n$ if 
$\{K_n\cap B\}_{n \in \n}\to K_0 \cap B$ in the Hausdorff distance for any closed ball $B \subset \r^n.$ If $K_j\subset K_{j+1}^\circ \subset K_0$ $\forall j\in\n$ and $\{K_j\}_{j \in \n} \to K_0$ in the Hausdorff topology, we simply write $\{K_j\}_{j \in \n} \nearrow K_0.$ Likewise we put $\{K_j\}_{j \in \n} \searrow K_0$ provided that $ K_0\subset K_{j+1}\subset K_j^\circ$ $\forall j\in\n$ and $\{K_j\}_{j \in \n} \to K_0$ in the Hausdorff topology.

The following theorem follows from classical Minkowski's Theorem \cite{min} (see also \cite{meeksyau}):

\begin{theorem}\label{th:mink}
Let $C$ be a (possibly neither bounded nor regular) convex domain of $\r^n.$ 

Then there exists a sequence $\{C_k\}_{k\in\n}$ of   bounded strictly convex analytic domains in $\r^n$ with $\{\overline{C_k}\}_{k \in \n} \nearrow \overline{C}.$ 

If in addition $C$ is bounded, then there exists a sequence $\{D_k\}_{k\in\n}$ of   bounded strictly convex analytic domains in $\r^n$ with $\{\overline{D_k}\}_{k \in \n} \searrow \overline{C}.$ \end{theorem}
Recall that a  convex domain $\mathcal{D}$ is said to be analytic if $\partial \mathcal{D}$ is an analytical hypersurface of $\r^n.$ 

\begin{definition} \label{de:proper}
Let $C$ be a (possibly neither bounded nor regular) convex domain  in $\r^n.$ A sequence $\{C_k\}_{k \in \n}$ of   convex domains in $\r^n$ is said to 
be {\em proper in $C$} if $C_k$ is bounded regular strictly  convex for all $k,$  $\{\overline{C_k}\}_{k \in \n} \nearrow \overline{C}$ in the Hausdorff topology and $\sum_{k \in \n} \sqrt{\frac{{\rm dist}(C_k,\partial C_{k+1})}{\kappa(C_k)}}=+\infty.$
\end{definition}
\begin{lemma} \label{lem:sucprop}
Any convex domain in $\r^n$ admits a proper sequence of convex domains.
\end{lemma}
\begin{proof}
Let $C$ be a convex domain  in $\r^n.$ By Theorem \ref{th:mink}, there exists a sequence $\{C_j\}_{j\in\n}$ of   bounded strictly convex analytic domains in $\r^n$ with $\{\overline{C_j}\}_{j \in \n} \nearrow \overline{C}.$ For the sake of simplicity write $d_j={\rm dist}(C_j,\partial C_{j+1})$ and $\kappa_j=\kappa(C_j)$ for all $j.$

Recall that $\frac{6}{\pi^2}\sum_{a \in \n} \frac1{a^2}=1,$ and for each $j \in \n$ fix  $m_j\in \n$ large enough so that $\frac{\sqrt{6}}{\pi} \sqrt{\frac{d_j}{\kappa_j}} \sum_{a=1}^{m_j}  \frac1{a}\geq 1.$ 
Call $d_{a,j}=d_j\frac{6}{\pi^2}\sum_{h=1}^a \frac1{h^2},$ set $C_{a,j}:=(C_j)_{d_{a,j}}$ (see Definition \ref{def:dt}),  $a=1,\ldots,m_j,$ and make the convention $C_{0,j}=C_j.$   
It is clear that $d_{a,j}<d_j,$   $C_{a,j}$ is analytical and strictly convex, $\overline{C_{a,j}} \subset C_{a+1,j} \subset \overline{C_{a+1,j}} \subset C_{j+1},$   and  ${\rm dist}(C_{a,j},\partial C_{a+1,j})=\frac{6d_j}{\pi^2 (a+1)^2}$ for all $a=0,\ldots,m_j-1.$ Furthermore, since $\kappa(C_{a,j})=\frac{\kappa_j}{1+d_{a,j} \kappa_j}\leq \kappa_j$ for all $a=0,\ldots,m_j-1,$ $$\sum_{a=0}^{m_j-1} \sqrt{\frac{{\rm dist}(C_{a,j},\partial C_{a+1,j})}{\kappa(C_{a,j})}} \geq  \frac{\sqrt{6}}{\pi} \sqrt{\frac{d_j}{\kappa_j}}\sum_{a=1}^{m_j} \frac1{a}\geq 1.$$

Let $\{D_k\}_{k \in \n}$ denote the enumeration of $\cup_{j \in \n} \{C_{a,j}\,|\, a=0,\ldots,m_j\}$ so that $\overline{D_{k}} \subset D_{k+1}$ for all $k.$  Since 
$$\sum_{k \in \n} \sqrt{\frac{{\rm dist}(D_k,\partial D_{k+1})}{\kappa(D_k)}}\geq \sum_{j \in \n} \left( \sum_{a=0}^{m_j-1} \sqrt{\frac{{\rm dist}(C_{a,j},\partial C_{a+1,j})}{\kappa(C_{a,j})}}\right)=+\infty,$$ $\{D_k\}_{k \in \n}$ is proper in $C$ and we are done.
\end{proof}

\subsubsection{$\curvo{\,,\,}$ and convex domains in $\c^3$}

Given $C\subset \c^3,$ we denote by ${\rm span}_\r(C)=\{\sum_{j=1}^n r_j v_j\,|\, r_j\in \r,\, v_j \in C,\, n \in \n\}$ and ${\rm span}_\c(C)={\rm span}_\r(C)+J({\rm span}_\r(C)),$  where $J:\c^3 \to \c^3,$ $J(v)=\imath v,$  is the usual complex structure. If $V\subset \c^3$ is a {\em real} subspace, the complex subspace $V_\c:=V \cap J(V)$ is said to be the {\em complex kernel} of $V.$

A real or complex vectorial hyperplane $V \subset \c^3$ is said to be {\em $\curvo{ \,,\, }$-degenerate} if $\curvo{ \,,\, }|_{V_\c \times V_\c}$ is a degenerate complex bilinear 1-form, that is to say, if  $V_\c=\curvo{ u}^\bot$ for some null vector $u.$ 
If $H = \esca{\nu}^\bot,$ $\nu \in\c^3-\{\vec{0}\},$  then $H_\c=\curvo{ \overline{\nu}}^\bot$ and $H$ is  $\curvo{ \,,\, }$-degenerate if and only if $\nu$ is null. If $\nu$ is not null, there exists a $\curvo{ \,,\,}$-conjugate basis $\{u_1,u_2,u_3\}$ of $\c^3$ so that $u_3=\overline{\nu}$ and ${\rm span}_\c(\{u_1,u_2\})=H_\c.$ 

\begin{definition} \label{de:regular}
A domain $\Omega \subset \c^3$ with non-empty $\partial \Omega$ is said to be $\curvo{ \,,\, }${\em -regular} if it is regular (i.e., with smooth $\partial \Omega$) and the real tangent space $T_p \, \partial \Omega$ is  not $\curvo{ \,,\, }$-degenerate for almost every $p \in \partial \Omega.$
\end{definition}

Given a regular convex domain $\Dcal \subset \c^3$ and a point $p \in \c^3 -\Dcal_{-1/\kappa(\Dcal)},$ we denote by 
\[
\Theta_\Dcal(p)=\;\doble{ \nu_\Dcal(p) }^\bot \cap \Theta\;\subset T_{\pi_\Dcal(p)} \, \partial \Dcal.
\]
\begin{definition}\label{def:wsc} 
A regular convex domain $\Dcal \subset \c^3 \equiv \r^6$ is said to be {\em null strictly convex} if $\mathcal{E}_\Dcal \cap \Theta_\Dcal(p) \neq \emptyset$ for all $p \in \partial \Dcal$ (see Definition \ref{def:escapingvector}).     
\end{definition}
This occurs, for instance, if for any $p \in \partial \Dcal$ there is $v \in \Theta_\Dcal(p) \cap \s^{5}$ such that  $\kappa_\Dcal(p,v)>0.$

\begin{claim}  \label{as:nsc}
The $\curvo{ \,,\, }$-regularity of domains and the null strictly convexity of regular convex domains of $\c^3$ are preserved by complex orthogonal transformations. 
\end{claim} 
\begin{proof}
From \eqref{eq:ATheta}, $A(\Theta)=\Theta$ and $A(\esca{\nu}^\bot)=\esca{\overline{A} \nu}^\bot$ for any $A \in \mathcal{O}(3,\c)$ and $\nu \in \c^3-\{\vec{0}\}.$ This shows that the  $\curvo{\,,\,}$-degeneracy of real (or complex) hyperplanes is preserved by complex orthogonal transformations and guarantees the assertion on the $\curvo{\,,\,}$-regularity.

For the null strictly convexity observe that
if  $A \in \mathcal{O}(3,\c),$  $p \in \partial \Dcal$  and $v \in \Theta_\Dcal(p)\cap \mathcal{E}_\Dcal,$  then $A(v) \in  \Theta_{A(\Dcal)} (A(p))\cap \mathcal{E}_{A(\Dcal)}.$ Indeed, it is clear that $A(p) \in \partial A(\Dcal)$ and $\|\overline{A} (\nu_\Dcal(p))\|\nu_{A(\Dcal)}(A(p))=\pm \overline{A} (\nu_\Dcal(p)).$ Therefore, 
\begin{equation}\label{eq:ortoe}
\doble{ \nu_{A(\Dcal)}(A(p)) }^\bot=A(\doble{ \nu_\Dcal (p)}^\bot)
\end{equation}
 and  $A(v) \in \Theta_{A(\Dcal)}(A(p)).$ Finally, since $\dist(p,q) \leq \|A^{-1}\|\dist(A(p),A(q))$ for all $p,$ $q \in \c^3,$ then $A(v) \in \mathcal{E}_{A(\Dcal)}$ as well and we are done. 
\end{proof} 

To finish this subsection, let us discuss the relationship between convex domains as those in the Main Theorem of this paper, and the concepts of $\curvo{ \,,\, }$-regularity and null strictly convexity. See Definition \ref{def:wide} for a well understanding of the following.

First notice that a strictly increasing sequence  $\rho=\{\rho_i\}_{1\leq i\leq n} \subseteq \{1,\ldots,6\}$  is wide if and only if $\dim_\c({\rm span}_\c(\r^\rho))\geq 2.$ 

\begin{proposition} \label{pro:dense}
Let $\rho$ be a wide sequence in $\{1,\ldots,6\},$ and let  $\Omega\subset \r^\rho$ be a regular strictly convex domain with $\kappa(\Omega)<+\infty.$  
Then $\mathcal{C}_\rho(\Omega)$ is $\curvo{ \,,\, }$-regular, null strictly convex and $\kappa(\mathcal{C}_\rho(\Omega))<+\infty.$ 
\end{proposition} 
\begin{proof} Label $n$ as the length of $\rho.$ Take $p \in \partial \mathcal{C}_\rho(\Omega),$ and observe that $\Pi_\rho (\nu_{\mathcal{C}_\rho(\Omega)}(p))=\nu_\Omega(\Pi_\rho(p))$ and 
$\Pi_{\rho^*} (\nu_{\mathcal{C}_\rho(\Omega)}(p))=\vec{0}.$ Call $\s_\rho^{n-1}=\s^5 \cap \r^\rho$ and $\Theta_{\rho}= \Theta \cap \s_\rho^{n-1}.$

Observe that $\Theta_{\rho}$ is an analytical submanifold of $\s_\rho^{n-1}$ of dimension $<n-1,$ and so has measure zero in $\s^{n-1}_\rho.$ Indeed, reason by contradiction and suppose that $\Theta_{\rho}= \s_\rho^{n-1},$ that is to say, $\r^\rho \subset \Theta.$ Then ${\rm span}_\c(\r^\rho) \subset \Theta$ as well, and so $\Theta$ contains a complex hyperplane of $\c^3$ (recall that $\rho$ is wide), a contradiction. The strictly convexity of  $\Omega$ implies that $\nu_\Omega:\partial \Omega \to \s^{n-1} \equiv \s_\rho^{n-1}$ is an injective local diffeomorphism, hence  $\nu_{\mathcal{C}_\rho(\Omega)}(p)$ is  not null for almost every $p \in \mathcal{C}_\rho(\Omega),$ proving the $\curvo{\,,\, }$-regularity.  

For the null strictly convexity, let us show first that $\Pi_\rho(\Theta_{\mathcal{C}_\rho(\Omega)}(p))\neq \{\vec{0}\}$ for all  $p \in \partial \mathcal{C}_\rho(\Omega).$  Indeed, if $\nu_{\mathcal{C}_\rho(\Omega)}(p)$ is null then $\overline{\nu_{\mathcal{C}_\rho(\Omega)}(p)} \in \Theta_{\mathcal{C}_\rho(\Omega)}(p)$ and  $\Pi_\rho(\overline{\nu_{\mathcal{C}_\rho(\Omega)}(p)})=\overline{\nu_\Omega(\Pi_\rho(p))}\neq \vec{0}.$ Assume now that  $\nu_{\mathcal{C}_\rho(\Omega)}(p)$ is not null, and reasoning by contradiction suppose that  $\Theta_{\mathcal{C}_\rho(\Omega)}(p) \subset \r^{\rho^*}.$ Since $\doble{\nu_{\mathcal{C}_\rho(\Omega)}(p)}^\bot$ is not $\curvo{ \,,\, }$-degenerate, then $\Theta_{\mathcal{C}_\rho(\Omega)}(p)$ contains two $\c$-linearly independent null vectors. This shows that  $\dim_\c((\r^{\rho^*})_\c)\geq 2,$ contradicting that $\rho$ is  wide. To finish, take any $v \in \Theta_{\mathcal{C}_\rho(\Omega)}(p)$ such that $\Pi_\rho(v) \neq \vec{0}.$ By the strictly convexity of $\Omega$ we have that  $\Pi_\rho(v) \in \mathcal{E}_\Omega,$ and so $v \in \mathcal{E}_{\Ccal_\rho(\Omega)}.$ 

Since $\kappa(\mathcal{C}_\rho(\Omega))=\kappa(\Omega)<+\infty,$ we are done.  
\end{proof}
\begin{remark} \label{re:falla}
If $\rho=\{2j-1,2j\}$ for some $j\in \{1,2,3\}$ and  $\Omega\subset \r^\rho$ is a regular convex domain,  then $\mathcal{C}_\rho(\Omega)$ is not  null strictly convex. Moreover, regular convex domains whose boundary contains a real half hyperplane  of $\c^3\equiv \r^6$ are not null strictly convex as well. 
\end{remark}


\section{Main Theorem}

The following Lemma, which will be proved later in Section \ref{sec:lema}, is the kernel of the proof of our main theorem. 

\begin{lemma}\label{lem:main}
Let $\mathcal{D}$ be a $\curvo{ \,,\, }$-regular and null strictly convex domain in $\c^3$ with $\kappa(\Dcal)<+\infty,$  and consider $r\in (0,1/\kappa(\Dcal)).$  Let $M$ be a Runge compact region in $\Ncal,$ $P_0\in M^\circ$ and $F\in{\sf N}(M)$ satisfying that:
\begin{equation}\label{eq:lema}
F(\partial M) \subset \mathcal{D}-\overline{\Dcal_{-r}}.
\end{equation}
Then, for any  regular convex domain $\hat{\Dcal}$  and $\epsilon>0$ such that  $\overline{\Dcal}\subset \hat{\Dcal}_{-\epsilon}\subset \overline{\hat{\Dcal}} \subset  \Dcal_{1/\epsilon},$  there exist a Runge compact region $\hat{M}\subset\Ncal$ and $\hat{F}\in{\sf N}(\hat{M})$ satisfying that:
\begin{enumerate}[\rm (i)]
\item $M\subset\hat{M}^\circ$ and $M$ is isotopic to $\hat{M}.$ 
\item $\|\hat{F}-F\|_1<\epsilon$ on $M,$
\item $\hat{F}(\partial\hat{M})\subset  \hat{\mathcal{D}}-\overline{\hat{\Dcal}_{-\epsilon}},$
\item $\hat{F}(\hat{M}-M^\circ)\subset \hat{\mathcal{D}}-\overline{\mathcal{D}_{-r}},$
\item $\dist_{(\hat{M},\hat{F})}(P_0,\partial \hat{M})> \dist_{(M,F)}(P_0,\partial (M))+\sqrt{\frac{d}{\kappa(\mathcal{D})}},$ where $d=\dist(\partial \Dcal,\partial \hat{\Dcal}).$
\end{enumerate}
\end{lemma}

We are now ready to state and prove our main theorem. See Definition \ref{def:wide} for notations.

\begin{theorem}\label{th:main}
Let $\rho\subset \{1,\ldots,6\}$ be a wide sequence, and let $\Omega$ be a (possibly neither bounded nor regular) convex domain in $\r^\rho.$  
Let $M\subset\Ncal$ be a Runge compact region in $\Ncal,$ and consider a null curve $X\in {\sf N}(M)$ satisfying that  
\begin{equation}\label{eq:theorem}  
\text{$\Pi_\rho( X(\partial M))\subset \Lambda-\overline{\Lambda_{-r}},$}
\end{equation}
where $\Lambda$ is a bounded regular strictly convex domain in $\r^\rho$ so that  $\overline{\Lambda}\subset \Omega$ and $r\in(0,1/\kappa(\Lambda)).$

Then, for any $\xi>0$ there exist an open  domain $N\subset \Ncal$ and a null curve $Y:N\to \c^3$ satisfying that
\begin{enumerate}[\rm (a)]
\item $M\subset N,$ $N$ is Runge and $N$ is isotopic to $\Ncal,$ 
\item $\|Y-X\|_1\leq \xi$ on $M,$
\item $Y$ is complete,
\item $\Pi_\rho\circ Y:N\to \Omega$ is proper  and $\Pi_\rho(Y(N-M^\circ))\subset \Omega-\overline{\Lambda_{-r}}.$
\end{enumerate}
\end{theorem}

\begin{proof}
Fix an auxiliary exhaustion $\{M_j\}_{j\in\n}$ of $\Ncal$ by  Runge compact regions so that $M_1^\circ$ is a tubular neighborhood of $M,$ and $M_{j-1}\subset M_j^\circ$   and the Euler characteristic $\chi(M_j-M_{j-1}^\circ)\in\{-1,0\}$ for all $j\geq 2.$

Let $\{\Omega_j\}_{j \in \n}$ be a proper sequence in $\Omega$ of convex domains (see Lemma \ref{lem:sucprop}), and denote by 
\[
d_j={\rm dist}(\Omega_j,\Omega_{j+1})\quad\text{ and }\quad \kappa_j=\kappa(\Omega_j).
\]
Without loss of generality, we can suppose that $\overline{\Lambda} \subset \Omega_1.$ Call 
\[\Omega_0=\Lambda\quad\text{ and }\quad\Dcal^j=\mathcal{C}_\rho(\Omega_j)=\{x\in\r^6 \;|\; \Pi_\rho(x)\in\Omega_j\}
\]
for all $j\geq 0.$ Since $\Omega_j$ is a bounded regular strictly convex domain, then $\kappa(\Omega_j)<+\infty$ and, by Proposition \ref{pro:dense},  $\Dcal^j$ is  $\curvo{ \,,\, }$-regular, null strictly convex and $\kappa(\Dcal^j)<+\infty,$ $j \geq 0.$

Without loss of generality, we assume that $\xi$ is small enough so that $\epsilon_0:=r-\xi>0$ 
and 
\begin{equation}\label{eq:3.2'}
\Pi_\rho( X(\partial M))\subset \Lambda-\overline{\Lambda_{-\epsilon_0}},
\end{equation}
see \eqref{eq:theorem}. 

Fix $Q_0\in M^\circ$ and denote by $N_0=M,$ $\sigma_0={\rm Id}_{N_0}$ and $X_0=X.$

Let us show the following 
\begin{claim}\label{cla:sucesion}
There exists a sequence $\{(N_j,\sigma_j,X_j,\epsilon_j)\}_{j\in\n},$  where $N_j\subset \Ncal$ is a Runge compact region isotopic to $M_j,$ $\sigma_j:N_j \to M_j$ is an isotopic homeomorphism, $X_j\in {\sf N}(N_j)$ and $\epsilon_j>0$ for all $j,$  satisfying that:
\begin{enumerate}[\rm (I{$_j$})]
\item $N_{j-1}\subset N_j^\circ$ and  $\sigma_j|_{N_{j-1}}=\sigma_{j-1}$  $\forall j\in \n,$
\item $\|X_j-X_{j-1}\|_1<\epsilon_j$ on $N_{j-1},$ $\forall j\in \n,$ where $\epsilon_j>0$ is  chosen small enough so that
\begin{equation}\label{eq:epsilonxi}
\epsilon_j<\min\left\{\epsilon_{j-1} \;,\; \frac{\xi}{2^{j+1}}\;,\; \frac1{2^{j+1}} \min\big\{\min_{N_k} \|d X_k/\sigma_\Ncal\|\;|\; k=0,\ldots,j-1\big\} \right\}\;\; \text{and}
\end{equation}
\begin{equation}\label{eq:epsilonxi2}
\overline{\Dcal^{j-1}}\subset(\Dcal^{j})_{-\epsilon_j}\subset \overline{\Dcal^j}\subset (\Dcal^{j-1})_{1/\epsilon_j},\quad \forall j\in\n,
\end{equation}
\item $X_j(\partial N_j)\subset \Dcal^j-\overline{(\Dcal^j)_{-\epsilon_j}},$ $\forall j\in\n,$ 
\item $X_j(N_j-N_{j-1}^\circ)\subset \Dcal^j-\overline{(\Dcal^{j-1})_{-\epsilon_{j-1}}},$ $\forall j\in \n,$ and
\item $\dist_{(N_j,X_j)}(Q_0,\partial N_j)> \dist_{(N_1,X_1)}(Q_0,\partial N_1)+ \sum_{a=1}^{j-1}  \sqrt{\frac{d_a}{\kappa_a}},$ $j\geq 2.$
\end{enumerate}
\end{claim}
\begin{proof}
The sequence will be constructed recursively. For $j=1$ choose $\epsilon_1$ satisfying \eqref{eq:epsilonxi} and \eqref{eq:epsilonxi2}. Observe that \eqref{eq:3.2'} allows to apply Lemma \ref{lem:main} for the data $(\Dcal,r,M,F,\hat{\Dcal},\epsilon,P_0)=(\Dcal^0,\epsilon_0,N_0, X_0,\Dcal^1,\epsilon_1,Q_0).$  Denote by $(N_1,X_1)$ the arising data $(\hat{M},\hat{F}).$
Note that $N_0 \subset N_1^\circ$ and that $N_1,$ $M_1$ and $N_0$ are homeomorphic. Choosing $\sigma_1:N_1 \to M_1$ any homeomorphism satisfying $\sigma_1|_{N_0}={\rm Id}_{N_0},$ all the above items hold for $j=1$ (item (V$_1$) makes no sense). 

Assume that we have already constructed $X_1,\ldots, X_{j-1}$ satisfying the required properties, and let us construct $X_j.$ Choose $\epsilon_j$ to satisfy \eqref{eq:epsilonxi} and \eqref{eq:epsilonxi2}, and let us distinguish two cases:

\noindent$\bullet$ $\chi(M_j-M_{j-1}^\circ)=0.$ We denote by $(N_j,X_j)$ the couple $(\hat{M},\hat{F})$ arising from Lemma \ref{lem:main} for the data
$(\Dcal,r,M,F,\hat{\Dcal},\epsilon,P_0)=(\Dcal^{j-1},\epsilon_{j-1},N_{j-1},X_{j-1},\Dcal^j,\epsilon_j,Q_0),$ which can be applied since (III$_{j-1}$) holds. Since $N_j,$ $N_{j-1}$ and $M_j$ are homeomorphic and $N_{j-1} \subset N_j^\circ,$ then we can choose $\sigma_j:N_j \to M_j$ any homeomorphism so that $\sigma_j|_{N_{j-1}}=\sigma_{j-1}.$  Properties (I$_j$), (II$_j$), (III$_j$) and (IV$_j$) are straightforward, whereas (V$_j$) follows from Lemma \ref{lem:main}-(v) and (V$_{j-1}$).

\noindent$\bullet$  $\chi(M_j-M_{j-1}^\circ)=-1.$  Consider a closed curve $\hat{\alpha}\in \Hcal_1(M_j,\z)-\Hcal_1(M_{j-1},\z)$ contained in $M_j^\circ$ and  intersecting $M_j-M_{j-1}^\circ$ in a Jordan arc $\alpha$ with  endpoints in $\partial (M_{j-1})$ and otherwise disjoint from $M_{j-1}.$ Since $M_j$ is Runge then  $\Hcal_1(M_j,\z)=\Hcal_1(M_{j-1}\cup \alpha,\z)$ and $\Ncal-(M_{j-1} \cup \alpha)$ has no bounded components. Take a Jordan arc $\gamma\subset  \Ncal-N_{j-1}^\circ$  and an isotopic homeomorphism $\varsigma:N_{j-1} \cup \gamma \to M_{j-1} \cup \alpha$ so that  $\varsigma|_{N_{j-1}}=\sigma_{j-1}$ and $\varsigma(\gamma)=\alpha.$ Note that $\Ncal -(N_{j-1} \cup \gamma)$ has no bounded components as well, hence without loss of generality we can assume that $S:=N_{j-1}\cup \gamma$ is  admissible.
 
Let $M_{j-1}'\subset M_j^\circ$ be any compact region isotopic to $M_j$ and containing $M_{j-1} \cup \alpha$ in its interior. Construct a generalized null curve $Z:S\to\c^3$ with $Z|_{N_{j-1}}=X_{j-1}$ and $Z(\gamma)\subset \Dcal^{j-1}-\overline{(\Dcal^{j-1})_{-\epsilon_{j-1}}}$ (this is possible by (III$_{j-1}$)).  Applying Lemma \ref{lem:runge} to $S,$ $Z$ and any tubular neighborhood  $W$ of $S,$ we get a compact region $N_{j-1}'$ and a null curve $X'_{j-1}\in {\sf N}(N_{j-1}')$ such that
\begin{itemize}
\item $S \subset (N_{j-1}')^\circ$ and $N_{j-1}'$ is isotopic to  $W,$ 
\item $\|X'_{j-1}-X_{j-1}\|_1<\epsilon_j/2$ on $N_{j-1},$
\item $X'_{j-1}(\partial N_{j-1}')\subset \Dcal^{j-1}-\overline{(\Dcal^{j-1})_{-\epsilon_{j-1}}},$ see (III$_{j-1}$), 
\item $X'_{j-1}(N'_{j-1}-N_{j-2}^\circ)\subset \Dcal^{j-1}-\overline{(\Dcal^{j-2})_{-\epsilon_{j-2}}},$ see (IV$_{j-1}$),  and
\item $\dist_{(N_{j-1}',X'_{j-1})}(Q_0,\partial N_{j-1}')>\dist_{(N_1,X_1)}(Q_0,\partial (N_1))+ \sum_{a=1}^{j-2}  \sqrt{\frac{d_a}{\kappa_a}},$ see (V$_{j-1}$).
\end{itemize}
Fix an isotopic homeomorphism $\sigma'_{j-1}:N_{j-1}'\to M'_{j-1}$ with $\sigma'_{j-1}|_S=\varsigma.$ To finish, notice that $\chi(M_j-(M_{j-1}')^\circ)=0,$  set $(N_j,X_j)$ as the  couple $(\hat{M},\hat{F})$ arising from Lemma \ref{lem:main} for the data $(\Dcal,r,M,F,\hat{\Dcal},\epsilon,P_0)=(\Dcal^{j-1},\epsilon_{j-1},N'_{j-1},X'_{j-1},\Dcal^j,\epsilon_j/2,Q_0),$ and take $\sigma_j:N_j \to M_j$ any homeomorphic extension of $\sigma'_{j-1}.$ 
\end{proof}
Label $N=\cup_{j\in\n} N_j$ and set $\sigma:N \to \Ncal,$ $\sigma|_{N_j}=\sigma_j.$ Since $\{M_j\}_{j\in\n}$ is an exhaustion of $\Ncal$ by Runge compact regions and $\sigma_j$ is an isotopic homeomorphism for all $j,$ then  $\sigma$ is an isotopic homeomorphim as well and  item (a) holds.

By items (II$_j$), $j\in\n,$ the sequence $\{X_j\}_{j\in\n}$ uniformly converges on compact subsets of $N$ to a holomorphic map $Y:N\to\c^3$ such that $\curvo{ dY,dY}=0$ and $\|Y-X\|_1\leq \xi$ on $M,$ which shows (b). 

Let us check that $Y$ is an immersion. Indeed, let $P\in N$ and take $j\in\n$ such that $P\in N_{j-1}.$ Then (II$_j$) and \eqref{eq:epsilonxi} imply that 
\begin{eqnarray*}
\|dY/\sigma_\Ncal\|(P) & \geq & \|dX_{j-1}/\sigma_\Ncal\|(P) - \sum_{k\geq j-1} \|X_{k+1}-X_k\|_1 \\
 & > & \|dX_{j-1}/\sigma_\Ncal\|(P) - \sum_{k\geq j} \epsilon_k \\
 & > & \|dX_{j-1}/\sigma_\Ncal\|(P) - \sum_{k\geq j} \frac{1}{2^{k+1}} \|dX_{j-1}/\sigma_\Ncal\|(P) \\
 & \geq & \frac12 \|dX_{j-1}/\sigma_\Ncal\|(P)>0,
\end{eqnarray*}
hence $Y$ is an immersion as claimed.

The completeness of $Y$ follows from (V$_j$), $j\in\n,$ and the fact that the series $\sum_{a\geq 1}  \sqrt{\frac{d_a}{\kappa_a}}$ is divergent (recall that $\{\Omega_j\}_{j \in \n}$ is a proper sequence in $\Omega$). 

Finally, let us check (d). Since (II$_k$), $k>j,$ we get that $\|Y-X_j\|_1\leq \xi/2^j$ on $N_j,$ and from (III$_j$) that $Y(\partial N_j) \subset (\Dcal^j)_{2^{-j} \xi} -\overline{(\Dcal^{j})_{-\epsilon_{j}- 2^{-j}\xi}}.$ Thus $\Pi_\rho (Y(\partial N_j)) \subset (\Omega_j)_{2^{-j} \xi} -\overline{(\Omega_j)_{-\epsilon_{j}- 2^{-j}\xi}},$ and so  by the maximum principle $\Pi_\rho (Y(N_j)) \subset (\Omega_j)_{2^{-j} \xi}$  for all $j.$ Therefore,  $\Pi_\rho (Y(N))\subset \overline{\Omega},$ hence  $\Pi_\rho (Y(N))\subset \Omega$ again by the maximum principle. From (IV$_j$) we deduce that $\Pi_\rho(Y(N_j-N_{j-1}^\circ)) \subset \Omega -\overline{(\Omega_{j-1})_{-\epsilon_{j-1}-2^{-j} \xi}}$ for all $j\geq 1,$ proving that $\Pi_\rho \circ Y:N \to \Omega$ is proper.  
Since $\Lambda \subset \Omega_j$ and $\epsilon_j+2^{-j-1}\xi<\epsilon_0+\xi=r$  then $\Lambda_{-r}\subset (\Omega_j)_{-\epsilon_j-2^{-j-1}\xi},$ and  so 
 $\Pi_\rho(Y(N_j-N_{j-1}^\circ)) \subset \Omega-\overline{\Lambda_{-r}}$ for all $j\geq 1,$  which proves (d) and the theorem.
\end{proof}

\begin{remark} \label{falla1}
The hypothesis that $\rho$ is wide and Proposition \ref{pro:dense} allow us to use Lemma \ref{lem:main} during the proof of Theorem \ref{th:main} (see Remark \ref{re:falla}). Hence, they play a crucial role in this setting. 

Moreover, recall that $\Ncal$ is hyperbolic, then so is $N.$ 
\end{remark}

Complete null curves in $\c^3$ project on complete holomorphic immersions in $\c^2$ and complete minimal immersions in $\r^3.$ From Theorem \ref{th:main} we infer that:

\begin{corollary}\label{co:main}
The following assertions hold:
\begin{enumerate}[\rm (a)]
\item For any convex domain $\Omega$ in $\c^3$ there exist a domain $N\subset \Ncal$ homeomorphic to $\Ncal$ and a complete proper null curve $F:N\to\Omega.$
\item For any convex domain $\Omega$ in $\r^3$ there exist a domain $N\subset \Ncal$ homeomorphic to $\Ncal$ and a conformal complete proper minimal immersion $F:N\to\Omega$ with vanishing flux.
\item For any convex domain $\Omega$ in $\c^2$ there exist a domain $N\subset \Ncal$ homeomorphic to $\Ncal$ and a complete proper holomorphic immersion $F:N\to\Omega.$ 
\item For any convex domain $\Omega$ in $\r^2$ there exist a domain $N\subset \Ncal$ homeomorphic to $\Ncal$ and a complete holomorphic immersion $F:N\to\c^2$ such that ${\rm Re}(F)(N)\subset\Omega$ and  ${\rm Re}(F):N\to\Omega$ is proper.
\item For any convex domain $\Omega$ in $\r^2$ there exist a domain $N\subset \Ncal$ homeomorphic to $\Ncal$ and a conformal complete minimal immersion $F=(F_j)_{j=1,2,3}:N\to\r^3$ with vanishing flux such that $(F_1,F_2)(N)\subset\Omega$ and $(F_1,F_2):N\to\Omega$ is proper.
\end{enumerate}
\end{corollary}
\begin{proof} Consider a bounded regular strictly convex domain $\Lambda$ with $\overline{\Lambda}\subset \Omega,$ and fix $r\in (0,1/\kappa(\Lambda))$ and $\xi>0.$   Set $\rho=\{1,2,3,4,5,6\},$ $\{1,3,5\}$ $\{1,2,3,4\},$ $\{1,3\}$ and $\{1,3\}$  in item (a), (b), (c), (d) and (e), respectively. In each case $\Omega \subset \r^\rho$ and we label $\Dcal=\Ccal_\rho(\Lambda).$

Let $H:\c\to\c^3$ be the properly embedded null curve given by $H(z)=(i z,z,\sqrt{2} z)^T,$ and note that $M:=H^{-1}(\Dcal_{-r/2})$ is a  a closed disc with  $H(\partial M) \subset \Dcal-\overline{(\Dcal)_{-r}}.$ Without loss of generality we can assume that $M \subset \Ncal.$ 

Let $Y:N \to \c^3$ be the null curve arising from Theorem \ref{th:main} for the data $\rho,$ $\Omega,$ $\Lambda,$ $r,$ $M,$ $X=H|_M$ and $\xi.$  The immersion $F=\Pi_\rho \circ Y$ solves items (a), (b) and (c). For items (d) and (e) choose $F=\Pi_{\hat{\rho}}\circ Y,$ where $\hat{\rho}=\{1,2,3,4\}$ and $\hat{\rho}=\{1,3,5\},$ respectively.
\end{proof}

A null curve $Z:M\to {\rm SL}(2,\c)$ is bounded and complete if an only if so is its Bryant's projection $\mathcal{B}(Z)=Z \cdot \bar{Z}^T.$  Furthermore, if $F:M\to \c^3$ is a complete bounded null curve such that $\overline{F(M)}\cap\{z_3=0\}=\emptyset,$ then $\mathcal{T}\circ F:M\to {\rm SL}(2,\c)$ is a complete bounded null curve as well (see \eqref{eq:corres} and \cite{muy1}).


\begin{corollary}\label{co:bryant}
The following assertions hold:
\begin{itemize}
\item There exist a domain $N\subset \Ncal$ homeomorphic to $\Ncal$ and a complete bounded null curve $Z:N\to{\rm SL}(2,\c).$ 
\item There exist a domain $N\subset \Ncal$ homeomorphic to $\Ncal$ and a conformal complete bounded CMC-1 immersion $X:N\to\h^3.$
\end{itemize}
\end{corollary}
\begin{proof} 
Use Corollary \ref{co:main}-(a) for an Euclidean ball  $\Omega$  in $\c^3$ whose closure is disjoint from $\{z_3=0\},$ and take into account the transformations $\mathcal{B}$ and $\mathcal{T}.$
\end{proof}

%
%


\section{Proof of Lemma \ref{lem:main}}\label{sec:lema}

The first step of the proof consists of constructing a special open covering $\Wgot$ of $\Dcal-\overline{\Dcal_{-r}}.$

Recall that  Remark \ref{re:pita} gives that  ${\rm dist}(p,\partial \hat{\mathcal D} \cap (p+\esca{\nu_\Dcal(p)}^\bot))\geq \sqrt{d^2+2 \frac{d}{\kappa(\mathcal{D})}}$ for all $p \in \Dcal-\overline{\Dcal_{-r}},$ where $d=\mbox{dist}(\partial \Dcal,\hat{\Dcal}).$ Take $\epsilon_1\in (0,\epsilon)$ small enough so that  
${\rm dist}(p,\partial \hat{\mathcal D} \cap (p+\esca{\nu_\Dcal(p)}^\bot)> \sqrt{\frac{d}{\kappa(\mathcal{D})}}+\epsilon_1$ for all $p \in \Dcal-\overline{\Dcal_{-r}}.$ Then, by a continuity argument and equation \eqref{eq:dista}, there exists an open neighborhood $V_p$ of $p$ in $\Dcal-\overline{\Dcal_{-r}}$ such that
\begin{equation} \label{eq:comple1}
\text{$\hat{V}_p \cap \overline{\mathcal{D}_{-r}}=\emptyset$ and  ${\rm dist}(q,\partial \hat{\mathcal D} \cap \hat{V}_p)>\sqrt{\frac{d}{\kappa(\mathcal{D})}}+\epsilon_1$  for all $q \in V_p$,}
\end{equation}
 where $\hat{V}_p=\cup_{(q_1,q_2) \in V_p\times V_ p} (q_1+\esca{\nu_\Dcal(q_2)}^\bot).$

Label $\Theta^* =\Theta \cap \s^5,$ and consider the continuous function $\s^5 \times \Theta^* \to [-1,1],$ $(\sigma,\theta)\mapsto \esca{\sigma,\theta}.$ Remark \ref{rem:nulidad} implies that $\mu_\sigma:=\max\{\esca{\sigma,\theta} \,|\, \theta \in \Theta^*\}>0$ for all $\sigma \in \s^5.$ Label $\mu=\min\{\mu_\sigma\,|\, \sigma \in \s^5 \}>0,$ and for each $p \in \Dcal-\overline{\Dcal_{-r}}$ choose $\theta_p \in \Theta^*$ and an open  neighborhood $U_p$ of $p$ in $\Dcal-\overline{\Dcal_{-r}}$ so that 
\begin{equation} \label{eq:up}
\text{$\esca{\nu_{\Dcal}(q),\theta_p}>\mu/2$ for all $q \in U_p.$}
\end{equation}
 
Finally call
\[
W_p=V_p \cap U_p\subset \Dcal-\overline{\Dcal_{-r}}
\]
for all $p \in \Dcal-\overline{\Dcal_{-r}},$ and set  $\Wgot=\{W_{p}\,|\, p \in \Dcal-\overline{\Dcal_{-r}}\}.$

In the second step of the proof, we are going to describe some subsets of $\Ncal-M^\circ$ and $\curvo{ \,,\,}$-conjugate bases in $\c^3$ associated to the open covering $\Wgot.$

Denote by $\alpha_1,\ldots,\alpha_k$ the connected components of $\partial M.$ For each $m \in \n$ let $\z_m=\{0,1,\ldots,m-1\}$  denote the additive cyclic group of integers modulus $m.$ From \eqref{eq:lema}, $\Wgot$ is an open covering of $F(\partial M),$ and so there exist $m \in \n,$ $m \geq 3,$  and a collection $\{\alpha_{i,j} \,|\, (i,j)\in \{1,\ldots,k\}\times \z_m\}$ such that for any $i \in \{1,\ldots,k\}:$

\begin{itemize}
\item  $\cup_{j=1}^{m} \alpha_{i,j}= \alpha_i,$  
\item $\alpha_{i,j}$ and $\alpha_{i,j+1}$ have a common endpoint $Q_{i,j}$ and are otherwise disjoint for all  $j\in \z_m,$ and
\item there exists  $W_{{i,j}}\in \Wgot$ such that  $F(\alpha_{i,j}\cup \alpha_{i,j+1})  \subset W_{{i,j}},$ for all  $j\in \z_m.$
\end{itemize}

Since $\Dcal$ is $\curvo{ \,,\, }$-regular, then we can find $p_{i,j}\in W_{i,j-1} \cap W_{i,j}$ such that $\esca{e_{i,j}}^\bot$ is not $\curvo{ \,,\,}$-degenerate, where 
\begin{equation}\label{eq:eijpij}
e_{i,j}= \nu_{\mathcal{D}}(p_{i,j}),\quad j\in \z_m
\end{equation}
(see Figure \ref{fig:alphas}).
\begin{figure}[ht]
    \begin{center}
    \scalebox{0.8}{\includegraphics{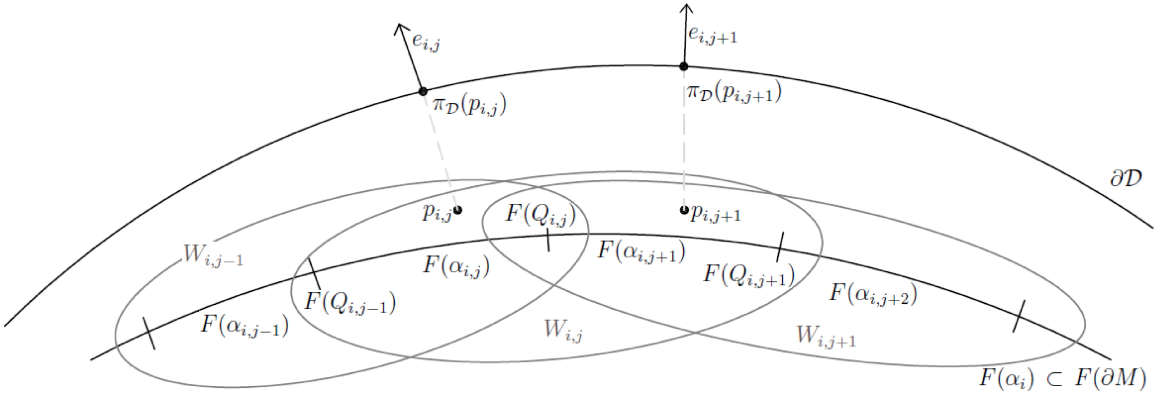}}
        \end{center}
        \vspace{-0.5cm}
\caption{The sets $W_{i,j}.$}\label{fig:alphas}
\end{figure}
Since $e_{i,j}$ is not null (see the discussion preceding Definition \ref{de:regular}), we can define 
$w_{i,j}=\overline{e_{i,j}}/\curvo{ \overline{e_{i,j}},\overline{e_{i,j}}},$ for all $i,$ $j,$
and notice that 
$
\doble{ e_{i,j} }^\bot=\curvo{ w_{i,j}}^\bot.
$ 
 Since $w_{i,j}$ is not null, we can take $u_{i,j},$ $v_{i,j} \in \curvo{ w_{i,j}}^\bot$ such that  $\{u_{i,j},v_{i,j},w_{i,j}\}$ is  $\curvo{ \,,\,}$-conjugate in $\c^3.$  Denote by $A_{i,j}$ the complex orthogonal matrix 
\begin{equation}\label{eq:defA}
A_{i,j}=(u_{i,j},v_{i,j},w_{i,j})^{-1}, 
\quad i\in \{1,\ldots,k\},\; j\in \z_m.
\end{equation}

Since $W_{i,j} \in \Wgot,$ \eqref{eq:up} yields  $\theta_{i,j} \in \Theta^*$ so that $\esca{\nu_{\Dcal}(q),\theta_{i,j}}>\mu/2$ for all $q \in W_{i,j}.$ In particular, 
\begin{equation}\label{eq:compa}
\text{$\esca{e_{i,j},\theta_{i,j}}$  and $\esca{e_{i,j+1},\theta_{i,j}}$ are both positive,  $i=1,\ldots,k,$ $j \in \z_m.$} 
\end{equation}
Notice that $\theta_{i,j}$ points to the exterior of $\Dcal$ at $\pi_\Dcal(q)$ for all $q \in W_{i,j}$ (recall that $\nu_\Dcal$ is the outward pointing normal on $\partial \Dcal$).

Let $\Rcal$ be a tubular neighborhood of $M$ in $\Ncal,$ and let $C_1,\ldots,C_k$ denote the finite collection of open annuli in $\Rcal-M,$  where  up to relabeling  $\alpha_i \subset \partial C_i,$ $i=1,\ldots,k.$ 

Let $\{r_{i,j} \subset \Rcal\,|\,i=1,\ldots,k, \; j\in \z_m\}$ be a collection of pairwise disjoint analytical Jordan arcs in $\Rcal$ such that $r_{i,j}$ has initial point $Q_{i,j}$ and $r_{i,j}-\{Q_{i,j}\} \subset C_i$ for all $i$ and $j.$ Label $T_{i,j}$ as the final point of $r_{i,j},$  and split $r_{i,j}$ into two subarcs   $s_{i,j}$ and $t_{i,j},$ where $Q_{i,j} \in s_{i,j}$ and $T_{i,j} \in t_{i,j}.$ In addition, choose these arcs so that  $S=M \cup(\cup_{i,j} r_{i,j})$ is admissible. See Figure \ref{fig:riemann}.

The third step of the proof deals with the existence of $H\in {\sf N}(\Rcal)$ {\em satisfying} the theses of the lemma {\em just on} $S.$ Roughly speaking, properties (a1), (a2), (a3) and (a4) in the following claim can be understood as the {\em restriction} to $S$ of (iv), (v), (iii) and (ii) in Lemma \ref{lem:main}, respectively.

\begin{claim} \label{cl:antena}
There exists $H\in{\sf N}(\Rcal)$ such that for any $(i,j) \in \{1,\ldots,k\} \times \z_m:$
\begin{enumerate}[\rm ({a}1)]
\item  $H(s_{i,j-1} \cup \alpha_{i,j}\cup s_{i,j}\cup  \alpha_{i,j+1} \cup s_{i,j+1})  \subset W_{{i,j}},$ 
\item if $J \subset s_{i,j}$ is a Borel measurable subset, then 
\begin{multline*} \label{eq:comple2}
\min\{L((A_{i,j} \cdot H|_{J})_3),L((A_{i,j+1} \cdot H|_{J})_3)\}+\\ \min\{L((A_{i,j} \cdot H|_{s_{i,j}-J})_3),L((A_{i,j+1} \cdot H|_{s_{i,j}-J})_3)\} > \max\{\|A_{i,j}\|,\|A_{i,{j+1}}\|\}\left(\sqrt{\frac{d}{\kappa(\mathcal{D})}}+\epsilon_1\right), 
\end{multline*}
where $(\,\cdot\,)_3$ means third (complex) coordinate and $L$ Euclidean length in $\c.$
\item $\esca{H(P)-H(Q_{i,j}),e_{i,j}},$  $\esca{H(P)-H(Q_{i,j}),e_{i,j+1}}>0$ for all $P \in t_{i,j},$ and $((H(T_{i,j})+\esca{e_{i,j}}^\bot)\cup (H(T_{i,j})+ \esca{e_{i,j+1}}^\bot)) \cap \overline{\Dcal_{1/\epsilon_1}}=\emptyset,$ and 
\item $\|H-F\|_1<\epsilon_1/(1+km)$ on $M.$
\end{enumerate}
\end{claim}
\begin{proof} Let $r_{i,j}(u),$ $u\in [0,1],$ be a smooth parameterization of $r_{i,j}$  so that $r_{i,j}([0,1/2])=s_{i,j}$ and $r_{i,j}([1/2,1])=t_{i,j}.$ Let  $\lambda_{i,j} \in \Theta$ be a null vector so that $(A_{i,j}(\lambda_{i,j}))_3,$ $(A_{i,j+1}(\lambda_{i,j}))_3\neq 0$ (see Remark \ref{rem:nulidad}) and the segment $\{\Lambda_{i,j}(s):=F(Q_{i,j})+s \lambda_{i,j}\,|\, s \in [0,1]\}$ lies in the interior of $W_{i,j-1} \cap W_{i,j}\cap W_{i,j+1},$ and set $\Lambda_{i,j}^*(s)=\Lambda_{i,j}(1-s),$ $s \in [0,1].$ Take $N \in \n$ large enough so that 
\begin{equation}\label{eq:a2}
2 N \min\{|(A_{i,j}(\lambda_{i,j}))_3|,|(A_{i,j+1}(\lambda_{i,j}))_3|\}>\max\{\|A_{i,j}\|,\|A_{i,j+1}\|\}\left(\sqrt{\frac{d}{\kappa(\mathcal{D})}}+\epsilon_1\right)
\end{equation}
 and 
\begin{equation} \label{eq:a3}
((F(Q_{i,j})+N \theta_{i,j}+\esca{e_{i,j}}^\bot)\cup (F(Q_{i,j})+N \theta_{i,j}+ \esca{e_{i,j+1}}^\bot)) \cap \overline{\Dcal_{1/\epsilon_1}}=\emptyset.
\end{equation}
For \eqref{eq:a3} take into account \eqref{eq:compa}.  Set $d_{i,j}:[0,1] \to \c^3,$ 
\begin{itemize}
\item $d_{i,j}(u)=\Lambda_{i,j}(8 N u-b+1)$ if $u \in [\frac{b-1}{8N},\frac{b}{8 N}]$ and  $b\in \{1,\ldots,2 N\}$ is odd,
\item $d_{i,j}(u)=\Lambda_{i,j}^*(8 N u-b+1)$ if $u \in [\frac{b-1}{8N},\frac{b}{8 N}]$ and  $b\in \{1,\ldots,2 N\}$ is even,
\item $d_{i,j}(u)=F(Q_{i,j})+ \xi_{i,j} (4 u-1)\theta_{i,j}$ if $u \in [1/4,1/2],$ where $\xi_{i,j}>0$  is small enough so that $d_{i,j}([1/4,1/2]) \subset W_{i,j-1}\cap  W_{i,j}\cap W_{i,j+1},$ and
\item $d_{i,j}(u)=F(Q_{i,j})+(\xi_{i,j}+ N (2 u-1))\theta_{i,j}$ if $u \in [1/2,1].$
\end{itemize}
The curves $d_{i,j}$ are continuous,  weakly differentiable and satisfy that $\curvo{d_{i,j}'(u),d_{i,j}'(u)}=0.$ Then, up to replacing $H|_{r_{i,j}}$ for $d_{i,j}$ for all $i,$ $j,$ items (a1), (a2) and (a3) formally hold. To finish, approximate $d_{i,j}$ by a smooth curve $c_{i,j}$ matching smoothly with $F$ at $Q_{i,j},$ and so that the map $\tilde{H}:S \to \c^3$  given by $\tilde{H}|_{M}=F,$ $\tilde{H}|_{r_{i,j}}(u)=c_{i,j}(u)$ for all $u \in [0,1],$ $i$ and $j,$ is a generalized null curve satisfying all the above items. Indeed, if $c_{i,j}$ is chosen close enough to $d_{i,j},$ (a1) is obvious, (a2) follows from \eqref{eq:a2}, and (a3) is an elementary consequence of \eqref{eq:a3}. 

The claim follows by a direct application of Lemma \ref{lem:runge} to $S,$  $\tilde{H}$ and $\Rcal.$ 
\end{proof}

To finish the proof of the lemma, we will suitably deform $H$ strongly on $\Rcal-S$ and hardly on $S.$ To do this, the introduction of some subsets in $\Rcal-M^\circ$ is required.
The fourth step of the proof is devoted to this issue.

Let $M_0$ be a compact region in $\Ncal$ such that $S-(\cup_{i,j} \{T_{i,j}\})\subset M_0^\circ\subset M_0\subset \Rcal,$ $\cup_{i,j} \{T_{i,j}\}\subset \partial M_0$  and $M_0-M^\circ$ consists of $k$ pairwise disjoint compact annuli $\hat{C}_1,\ldots,\hat{C}_k,$ where $\hat{C}_i \subset \overline{C_i}.$ Denote by $\Omega_{i,j}$ the closed disc in $M_0-M^\circ$ bounded by $\alpha_{i,j},$ $r_{i,j-1},$ $r_{i,j}$ and a piece, named $\beta_{i,j}$ of $\partial \hat{C}_i$ connecting $T_{i,j-1}$ and $T_{i,j}.$

Consider compact neighborhoods (with the topology of a closed disc) $\tilde{\alpha}_{i,j},$  $\tilde{s}_{i,j}$ and $\tilde{t}_{i,j}$ in $M_0-M^\circ$ of $\alpha_{i,j},$ $s_{i,j}$ and $t_{i,j},$ respectively,  $i=1,\ldots,k,$ $j \in \z_m,$  and satisfying the following properties (see Figure \ref{fig:riemann}):
\begin{enumerate}[({b}1)]
\item $\tilde{\alpha}_{i,j}\cap (\tilde{t}_{i,j-1} \cup \tilde{t}_{i,j})=\emptyset,$  $\tilde{\alpha}_{i,j}\cap (\tilde{s}_{i,a} \cup \tilde{t}_{i,a})=\emptyset,$ $a \neq j-1, j,$  $(\tilde{s}_{i,j-1}\cup\tilde{t}_{i,j-1}) \cap (\tilde{s}_{i,j}\cup\tilde{t}_{i,j})=\emptyset,$
\item $K_{i,j}:=\overline{\Omega_{i,j}-(\tilde{s}_{i,j-1}\cup \tilde{t}_{i,j-1}\cup \tilde{\alpha}_{i,j}\cup \tilde{s}_{i,j}\cup \tilde{t}_{i,j})}$ is a compact disc and $K_{i,j} \cap \beta_{i,j}$ is a Jordan arc disjoint from $\{T_{i,j-1},T_{i,j}\},$
\item  $H(\tilde{s}_{i,j-1}\cup \tilde{\alpha}_{i,j}\cup \tilde{s}_{i,j}\cup  \tilde{\alpha}_{i,j+1}\cup \tilde{s}_{i,j+1})  \subset W_{{i,j}},$ 

\item if $J \subset \tilde{s}_{i,j}$ is an arc connecting $\tilde{\alpha}_{i,j}\cup \tilde{\alpha}_{i,j+1}$ and $\tilde{t}_{i,j},$  $J_1=J \cap \Omega_{i,j}$ and $J_2=J \cap \Omega_{i,j+1},$ then 

\begin{equation*} \label{eq:comple2}
L((A_{i,j} \cdot H|_{J_1})_3)+L((A_{i,j+1} \cdot H|_{J_2})_3) > \max\{\|A_{i,j}\|,\|A_{i,{j+1}}\|\}\left(\sqrt{\frac{d}{\kappa(\mathcal{D})}}+\epsilon_1\right),
\end{equation*}
\item $\esca{H(P)-H(Q_{i,j-1}),e_{i,j}}>0$ $\forall P \in \tilde{t}_{i,j-1},$ $\esca{H(P)-H(Q_{i,j}),e_{i,j}}>0$ $\forall P \in \tilde{t}_{i,j},$  and $(H(Q)+\esca{e_{i,j}}^\bot)\cap \overline{\Dcal_{1/\epsilon_1}}=\emptyset$ $\forall Q\in (\tilde{t}_{i,j-1}\cup\tilde{t}_{i,j})\cap \partial M_0.$
\end{enumerate}

These choices are possible due to properties (a1), (a2) and (a3) and a continuity argument.

\begin{figure}[ht]
    \begin{center}
    \scalebox{0.55}{\includegraphics{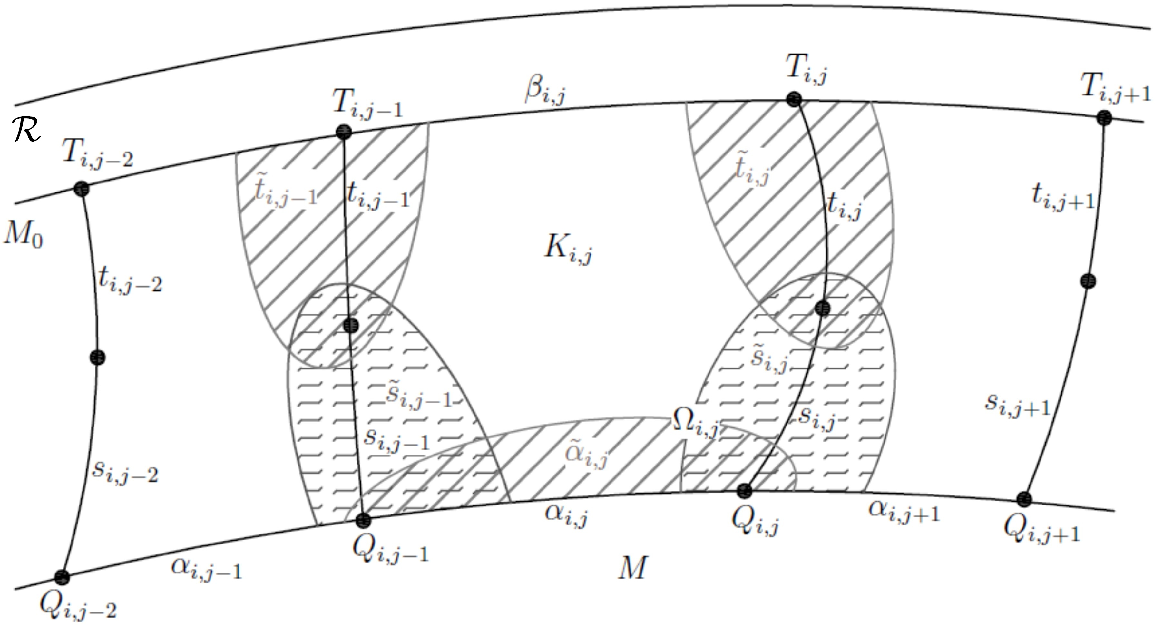}}
        \end{center}
        \vspace{-0.5cm}
\caption{$M_0$}\label{fig:riemann}
\end{figure}

The last step of the proof consists of {\em pushing} $H(K_{i,j})$ out of $\overline{\Dcal_{1/\epsilon_1}}$ in the direction of $\doble{e_{i,j}}^\bot$ for all $i$ and $j,$ hardly modifying $H$ on $M$ (see Claim \ref{cl:Hn} below). The choice of this direction allows to preserve the already achieved in properties (b4) and (b5), and taking also (b3) into account, to obtain properties (iii) and (iv) in the statement of Lemma \ref{lem:main}.

Let $\eta:\{1,\ldots,k m\}\to\{1,\ldots,k\}\times \z_m$ be the bijection $\eta(n)=(E(\frac{n-1}{m})+1,n-1),$ where $E(\cdot)$ means integer part. The process is enclosed in the following

\begin{claim}\label{cl:2ª}
There exists a sequence $H_0, H_1,\ldots, H_{km}$ of null curves in ${\sf N}(\Rcal)$ satisfying the following properties:
\begin{enumerate}[\rm ({c}1{$_n$})]
\item $H_n({s}_{\eta(j)-(0,1)}\cup \alpha_{\eta(j)}\cup {s}_{\eta(j)}\cup  \alpha_{\eta(j)+(0,1)}\cup {s}_{\eta(j)+(0,1)})  \subset W_{\eta(j)},$ $\forall j\in \{1,\ldots, km\},$

\item $H_n((\tilde{s}_{\eta(j)-(0,1)}\cup \tilde{\alpha}_{\eta(j)}\cup \tilde{s}_{\eta(j)})\cap \Omega_{\eta(j)})  \subset W_{\eta(j)}+\doble{ e_{\eta(j)} }^\bot,$ $\forall j\in\{1,\ldots, n\},$ $n \geq 1,$

\item $H_n((\tilde{s}_{\eta(j)-(0,1)}\cup \tilde{\alpha}_{\eta(j)}\cup \tilde{s}_{\eta(j)})\cap \Omega_{\eta(j)})  \subset W_{\eta(j)},$ $\forall j\in\{n+1,\ldots,km\},$

\item if $J \subset \tilde{s}_{\eta(j)}$ is an arc connecting $\tilde{\alpha}_{\eta(j)}\cup \tilde{\alpha}_{\eta(j)+(0,1)}$ and $\tilde{t}_{\eta(j)},$  $J_1 =J \cap \Omega_{\eta(j)}$ and $J_2 =J \cap \Omega_{\eta(j)+(0,1)},$ then
\begin{equation*}
L((A_{\eta(j)} \cdot H_n|_{J_1})_3)+ L((A_{\eta(j)+(0,1)} \cdot H_n|_{J_2})_3) > \max\{\|A_{\eta(j)}\|,\|A_{\eta(j)+(0,1)}\|\}\left(\sqrt{\frac{d}{\kappa(\mathcal{D})}}+\epsilon_1\right),
\end{equation*}
$\forall j\in \{1,\ldots, km\},$

\item $\esca{H_n(P)-H_n(Q_{\eta(j)}),e_{\eta(j)}}>0$ $\forall P \in \tilde{t}_{\eta(j)}\cap \Omega_{\eta(j)},$ $\forall j\in\{1,\ldots, km\},$ 

\item $\esca{H_n(P)-H_n(Q_{\eta(j)-(0,1)}),e_{\eta(j)}}>0$ $\forall P \in \tilde{t}_{\eta(j)-(0,1)}\cap \Omega_{\eta(j)},$ $\forall j\in\{1,\ldots, km\},$

\item $(H_n(Q)+\esca{e_{\eta(j)}}^\bot)\cap \overline{\Dcal_{1/\epsilon_1}}=\emptyset$ $\forall Q\in (\tilde{t}_{\eta(j)}\cup\tilde{t}_{\eta(j)-(0,1)})\cap \partial M_0\cap\Omega_{\eta(j)},$ $\forall j\in\{1,\ldots, km\},$

\item $H_n(K_{\eta(j)})\cap \overline{\Dcal_{1/ \epsilon_1}}=\emptyset,$ $\forall j\in\{1,\ldots, n\},$ $n\geq 1,$ and

\item $\|H_n-H_{n-1}\|_1<\epsilon_1/(km+1)$ on $\overline{M_0-\Omega_{\eta(n)}},$ $n\geq 1.$
\end{enumerate}
\end{claim}

Roughly speaking, the above properties mean that the null curve $H_n$ satisfies the theses of the lemma {\em just on} $M\cup (\cup_{j=1}^n\Omega_{\eta(j)}).$ Items (c9$_n$); (c7$_n$) and (c8$_n$); (c2$_n$), (c5$_n$) and (c6$_n$); and (c4$_n$) correspond to the restriction to $M\cup (\cup_{j=1}^n\Omega_{\eta(j)})$ of (ii); (iii); (iv); and (v), respectively. 
\begin{proof}[Proof of Claim \ref{cl:2ª}]
Take $H_0:=H.$ Notice that ({c}1{$_0$}), ({c}3{$_0$}), ({c}4{$_0$}), ({c}5{$_0$}), ({c}6{$_0$}) and ({c}7{$_0$}) follow from (b3), (b4) and (b5), whereas the remaining properties make no sense for $n=0.$

Reason by induction and assume we have constructed $H_0,\ldots,H_{n-1}$ satisfying all the above properties. The construction of $H_n$ is contained in the following

\begin{claim}\label{cl:Hn}
Let $\epsilon_0>0.$ There exists $H_n\in {\sf N}(\Rcal)$ such that
\begin{enumerate}[\rm ({d}1)]
\item $\|H_n- H_{n-1}\|_1< \epsilon_0 \|A_{\eta(n)}^{-1}\|$ on $\overline{M_0-\Omega_{\eta(n)}},$
\item $\doble{ H_n- H_{n-1},e_{\eta(n)}} \equiv 0,$ and
\item $H_n(K_{\eta(n)}) \cap (\overline{\Dcal_{1/\epsilon_1}})=\emptyset.$
\end{enumerate}
\end{claim}
\begin{proof}
Denote by $G=A_{\eta(n)}\cdot H_{n-1}\in{\sf N}(\Rcal)$ and write $G=(G_1,G_2,G_3)^T.$ 

Let $\gamma$ denote a Jordan arc in $\tilde{\alpha}_{\eta(n)}$ disjoint from $(\tilde{s}_{\eta(n)-(0,1)}\cup\tilde{t}_{\eta(n)-(0,1)}) \cup (\tilde{s}_{\eta(n)}\cup\tilde{t}_{\eta(n)}),$ with initial point $\Gamma_1\in\alpha_{\eta(n)}$ and final point $\Gamma_2\in\partial K_{\eta(n)}$ and otherwise disjoint from $\partial \tilde{\alpha}_{\eta(n)}.$ Choose $\gamma$ so that $dG_3|_{\gamma}$ never vanishes. Denote by $S_n$ the closure of $(M_0-\Omega_{\eta(n)})\cup \gamma\cup K_{\eta(n)},$ and without loss of generality assume that $S_n$ is admissible.

Recall that $A_{\eta(n)}\in \mathcal{O}(3,\c)$ implies that $A_{\eta(n)}(\doble{e_{\eta(n)}}^\bot)=\doble{\nu_{A(\Dcal)}(A_{\eta(n)}(\pi_\Dcal(p_{\eta(n)})))}^\bot$ (see \eqref{eq:eijpij} and \eqref{eq:ortoe}). From the definition of $A_{\eta(n)},$ $e_{\eta(n)}$ and $w_{\eta(n)}$ (see \eqref{eq:defA}), one has 
\[
A_{\eta(n)}(\doble{e_{\eta(n)}}^\bot)=A_{\eta(n)}(\curvo{w_{\eta(n)}}^\bot)= \{(z_1,z_2,0) \in \c^3\,|\, z_1,\,z_2 \in \c\}.
\]
Since $A_{\eta(n)} (\Dcal)$ is null strictly convex (see  Claim \ref{as:nsc})  and $ A_{\eta(n)} (\overline{\Dcal_{1/\epsilon_1}}) \subset  \overline{(A_{\eta(n)}(\Dcal))_{{\|A_{\eta(n)}\|}/{\epsilon_1}}},$ there exists a escaping vector $\zeta \in \mathcal{E}_{A_{\eta(n)}(\Dcal)} \cap \Theta \cap A_{\eta(n)}(\doble{ e_{\eta(n)}}^\bot)$ such that \begin{equation} \label{eq:fuera}
 \text{${\rm dist}(G(\Gamma_1)+ \zeta, A_{\eta(n)} (\overline{\Dcal_{1/\epsilon_1}}))>{\rm diam}(G(\gamma \cup K_{\eta(n)})),$}
\end{equation}
where ${\rm diam}(\cdot)$ means Euclidean diameter in $\c^3.$ 

Let $\gamma(u),$ $u \in [0,1],$ be a smooth parameterization of $\gamma$ with $\gamma(0)=\Gamma_1.$  Label  $\tau_j= \gamma([0,1/j])$ and consider the parameterization   $\tau_j(u)=\gamma(u/j),$ $u \in [0,1].$ Write $G_{3,j}(u)=G_3(\tau_j(u)),$  $u \in [0,1],$ and notice that $\frac{d G_{3,j}}{du}(0)=\frac1{j}\frac{d (G_{3}\circ \gamma)}{du}(0)$ for all $j \in \n.$

Since $e_{\eta(n)}$ is not null and $A_{\eta(n)}\in \mathcal{O}(3,\c),$ there is a null vector $\zeta^* \in A_{\eta(n)}(\doble{ e_{\eta(n)}}^\bot)$ so that $\{\zeta,\zeta^*\}$ is a basis of $A_{\eta(n)}(\doble{ e_{\eta(n)}}^\bot)$ and $\curvo{\zeta,\zeta^*} \neq 0.$

Set $\zeta_j=\zeta- \frac{(d G_{3,j}/du (0))^2}{2\curvo{\; \zeta,\zeta^*\;}} \zeta^*\in A_{\eta(n)}(\doble{ e_{\eta(n)}}^\bot),$ $j \in \n,$ and observe that $\lim_{j \to \infty} \zeta_j =\zeta$ and  $\curvo{ \zeta_j,\zeta_j}=-(\frac{d G_{3,j}}{du}(0))^2$ for all $j.$ Furthermore, by  \eqref{eq:fuera} we can also assume that ${\rm dist}(G(\Gamma_1)+\zeta_j, A_{\eta(n)} (\overline{\Dcal_{1/\epsilon_1}}))>{\rm diam}(G(\gamma \cup K_{\eta(n)}))$  for all $j \in \n.$

 Set $h_j:[0,1] \to \c^3,$  $h_j(u)=G(\Gamma_1)+ i \frac{G_{3,j}(u)-G_{3,j}(0)}{\curvo{ \;\zeta_j,\zeta_j\;}^{1/2}} \zeta_j + (0,0,G_{3,j}(u)-G_3(\Gamma_1)),$ and notice that $\curvo{ h_j'(u),h_j'(u) }=0$ and $\doble{ h_j'(u),h_j'(u) }$ never vanishes on $[0,1],$  $j\in \n.$ Up to choosing a suitable branch of $\curvo{ \zeta_j,\zeta_j}^{1/2},$ the sequence  $\{h_j\}_{j \in \n}$ converges uniformly on $[0,1]$ to $h_\infty:[0,1] \to \c^3,$  $h_\infty(u)=u \zeta+ G(\Gamma_1).$ Since ${\rm dist}(h_\infty(1), A_{\eta(n)} (\overline{\Dcal_{1/\epsilon_1}}))>{\rm diam}(G(\gamma \cup K_{\eta(n)}))$ (see equation \eqref{eq:fuera}), then  
there exists $j_0 \in \n$ such that 
\begin{equation} \label{eq:fuera1}
\text{${\rm dist}(h_{j_0}(1), A_{\eta(n)} (\overline{\Dcal_{1/\epsilon_1}}))>{\rm diam}(G(\gamma \cup K_{\eta(n)})).$}  
\end{equation}

Set $\hat{h}:\tau_{j_0} \to \c^3,$ $\hat{h}(P)=h_{j_0}(u(P)),$ where $u(P)\in [0,1]$ is the only value for which $\tau_{j_0}(u(P))=P.$   Let $\hat{G}=(\hat{G}_1,\hat{G}_2,\hat{G}_3): S_n\to \c^3$ denote the continuous map given by
$$\hat{G}|_{\overline{M_0-\Omega_{\eta(n)}}}=G|_{\overline{M_0-\Omega_{\eta(n)}}},\; \hat{G}|_{\tau_{j_0}}=\hat{h},\; \hat{G}|_{(\gamma-\tau_{j_0})\cup K_{\eta(n)}}=G|_{(\gamma-\tau_{j_0})\cup K_{\eta(n)}}-G(\tau_{j_0}(1))+\hat{h}(\tau_{j_0}(1)).$$
Notice that $\hat{G}_3=G_3|_{S_n}.$ 
The equation $\curvo{ d\hat{G},d \hat{G}}=0$ formally holds except at the points $\Gamma_1$ and $\tau_{j_0}(1)$ where smoothness could fail. 
 Up to smooth approximation,  $\hat{G}$ is a generalized null curve  satisfying that 
\[
\hat{G}|_{\overline{M_0-\Omega_{\eta(n)}}}=G|_{\overline{M_0-\Omega_{\eta(n)}}},\; \hat{G}_3=G_3|_{S_n},\; \hat{G}(K_{\eta(n)}) \cap A_{\eta(n)}(\overline{\Dcal_{1/\epsilon_1}})=\emptyset.
\]

Applying Lemma \ref{lem:runge} to $S_n,$ $\Rcal$ and $\hat{G},$ we get a null curve $Z=(Z_1,Z_2,Z_3)^T\in {\sf N}(\Rcal)$ such that
\begin{equation}\label{eq:propZ}
\|Z- A_{\eta(n)}\cdot H_{n-1}\|_1<\epsilon_0 \text{ on }\overline{M_0-\Omega_{\eta(n)}},\;
Z_3=(A_{\eta(n)}\cdot H_{n-1})_3,\;
Z(K_{\eta(n)}) \cap A_{\eta(n)}(\overline{\Dcal_{1/\epsilon_1}})=\emptyset.
\end{equation}

Define $H_n:= A_{\eta(n)}^{-1}\cdot Z\in {\sf N}(\Rcal).$ From \eqref{eq:propZ}, $H_n$ trivially satisfies (d1) and (d3). Finally, (d2) follows from  \eqref{eq:propZ}, \eqref{eq:defA} and the definition of $w_{\eta(n)}.$
\end{proof}

Let us check that $H_n$ fulfills properties (c1$_n$) to (c9$_n$) provided that $\epsilon_0$ is chosen small enough. 

Item (c1$_{n-1}$) and (d1) imply (c1$_n$) for small enough $\epsilon_0.$ 
Item (c2$_n$) follows from (d1) and (c2$_{n-1}$) for $j<n$ provided that $\epsilon_0$ is small enough, and from (d2)  and (c3$_{n-1}$) for $j=n.$ 
Item (c3$_{n-1}$) and (d1) also give (c3$_n$) for small enough $\epsilon_0.$ 

Observe that (d1) shows that  $\|A_{\eta(j)}\cdot H_n- A_{\eta(j)}\cdot H_{n-1}\|<\epsilon_0 \|A_{\eta(n)}^{-1}\|\, \|A_{\eta(j)}\|$ and $\|A_{\eta(j)+(0,1)}\cdot H_n- A_{\eta(j)+(0,1)}\cdot H_{n-1}\|<\epsilon_0 \|A_{\eta(n)}^{-1}\|\, \|A_{\eta(j)+(0,1)}\|$ on $\Omega_{\eta(j)}$ $\forall j\neq n,$ whereas (d2) gives that $(A_{\eta(n)}\cdot H_n)_3= (A_{\eta(n)}\cdot H_{n-1})_3$ on $\Omega_{\eta(n)}.$ Then, taking into account (c4$_{n-1}$) we get (c4$_n$) provided that $\epsilon_0$ is small enough.

To prove (c5$_n$) (respectively, (c6$_n$), (c7$_n$)) we distinguish cases. If $j\neq n,$ use (d1) and (c5$_{n-1}$) (respectively, (c6$_{n-1}$), (c7$_{n-1}$)) for small enough $\epsilon_0,$  whereas for $j=n$ we use (d2) and (c5$_{n-1}$) (respectively, (c6$_{n-1}$), (c7$_{n-1}$)). Item (c8$_n$) follows from (c8$_{n-1}$), (d1) and (d3) for $\epsilon_0$ small enough. Finally, (c9$_{n}$) is an immediate consequence of (d1) for $\epsilon_0<\frac{\epsilon_1}{k m+1}\frac1{\|A_{\eta(n)}^{-1}\|}.$

The proof of Claim \ref{cl:2ª} is done.
\end{proof}

Finally, let us check that, up to a shrinking of its domain of definition, the null curve $H_{km}$ satisfies the conclusion of Lemma \ref{lem:main}.

Notice that 
$H_{km}(\partial M\cup (\cup_{j=1}^{km}s_{\eta(j)}))\subset \Dcal -\overline{\Dcal_{-r}}$
and $H_{km}({\partial M_0 \cup (\cup_{j=1}^{km}K_{\eta(j)}}))\cap \overline{\Dcal_{1/\epsilon_1}}=\emptyset.$
Indeed, use (c1$_{km}$) and that $\cup_{j=1}^{k m} W_{\eta(j)}\subset \Dcal-\overline{\Dcal_{-r}}$ for the first  assertion,  and (c7$_{km}$) and (c8$_{km}$) for the second one. Denote by $\hat{M}_0$ the connected component of $H_{km}^{-1}(\overline{\hat{\Dcal}})$ containing $M,$ and  observe that  
\begin{equation}\label{eq:H2}
H_{km}(\partial \hat{M}_0)\subset \partial \hat{\Dcal},
\end{equation}
\begin{equation}\label{eq:H1}
M\cup (\cup_{j=1}^{km}s_{\eta(j)})\subset \hat{M}_0^\circ \subset \hat{M}_0\subset M_0^\circ-\cup_{j=1}^{km}K_{\eta(j)},
\end{equation} 
and by the convex hull property  $\hat{M}_0-M^\circ$ consists of $k$ pairwise disjoint closed annuli.

Consider the null curve $\hat{F}_0:=H_{km}|_{\hat{M}_0}\in {\sf N}(\hat{M}_0),$ and observe that $\hat{F}_0$ satisfies the following properties:
\begin{enumerate}[\rm({P}1)]
\item $\|\hat{F}_0-F\|_1<\epsilon_1$ on $M.$ Use (a4) and (c9$_n$), $n=1,\ldots, km.$

\item $\hat{F}_0(\hat{M}_0-M^\circ)\subset \overline{\hat{\Dcal}}-\overline{\Dcal_{-r}}.$ From \eqref{eq:H2}, it suffices to check that $\hat{F}_0(\hat{M}_0-M^\circ)\cap \overline{\Dcal_{-r}}=\emptyset.$ Let $P\in \hat{M}_0-M^\circ.$ Taking into account \eqref{eq:H1}, there are three possible locations for the point $P:$
\begin{itemize}
\item $P\in  \tilde{s}_{\eta(n)}\cup \tilde{\alpha}_{\eta(n)}$ for some $n\in \{1,\ldots,km\}.$ Combining (c2$_{km}$) and \eqref{eq:comple1}, we get $\hat{F}_0(P)\notin \overline{\Dcal_{-r}}.$ 
\item $P\in \tilde{t}_{\eta(n)}\cap \Omega_{\eta(n)}$ for some $n\in\{1,\ldots,km\}.$ Use that $\hat{F}_0(Q_{\eta(n)})\in W_{\eta(n)}$ (see (c1$_{km}$)), $(\hat{F}_0(Q_{\eta(n)})+\esca{e_{\eta(n)}}^\bot)\cap \overline{\Dcal_{-r}}=\emptyset$ (see \eqref{eq:comple1}) and $\esca{\hat{F}_0(P)-\hat{F}_0(Q_{\eta(n)}),e_{\eta(n)}}>0$ (see (c5$_{km}$)), and recall that $e_{\eta(n)}$ is the outward pointing normal to $\partial \Dcal$ at $\pi_\Dcal (p_{\eta(n)}).$
\item $P\in \tilde{t}_{\eta(n)-(0,1)}\cap \Omega_{\eta(n)}.$ Use that $\hat{F}_0(Q_{\eta(n)-(0,1)})\in W_{\eta(n)},$ and argue as above but using (c6$_{km}$) instead of (c5$_{km}$).
\end{itemize}

\item $\dist_{(\hat{M}_0,\hat{F}_0)}(P_0,\partial \hat{M}_0)> \dist_{(M,F)}(P_0,\partial M)+\sqrt{\frac{d}{\kappa(\mathcal{D})}}.$  Taking into account (P1), it suffices to prove that $\dist_{(\hat{M}_0,\hat{F}_0)}(\partial M,\partial \hat{M}_0)>\sqrt{\frac{d}{\kappa(\mathcal{D})}}+\epsilon_1.$ Let  $P\in \partial \hat{M}_0$ and let $\gamma \subset \hat{M}_0-M^\circ$ be a connected curve connecting $\partial M$ and $P.$ There are two possible locations for the point $P:$ 
\begin{itemize}
\item If $P\in \tilde{t}_{\eta(n)}$ (see Figure \ref{fig:curvas1}) for some $n\in \{1,\ldots, km\},$ then by \eqref{eq:H1} there exists a sub-arc $\hat{\gamma}\subset \gamma\subset \tilde{s}_{\eta(n)}$ connecting $\tilde{\alpha}_{\eta(n)}\cup \tilde{\alpha}_{\eta(n)+(0,1)}$ and $\tilde{t}_{\eta(n)}.$ Thus, Remark \ref{rem:AF} and (c4$_{km}$)   give that  $\sqrt{d/\kappa(\Dcal)}+\epsilon_1<L(\hat{F}_0(\hat{\gamma}))\leq L(\hat{F}_0(\gamma)).$
\begin{figure}[ht]
    \begin{center}
    \scalebox{0.45}{\includegraphics{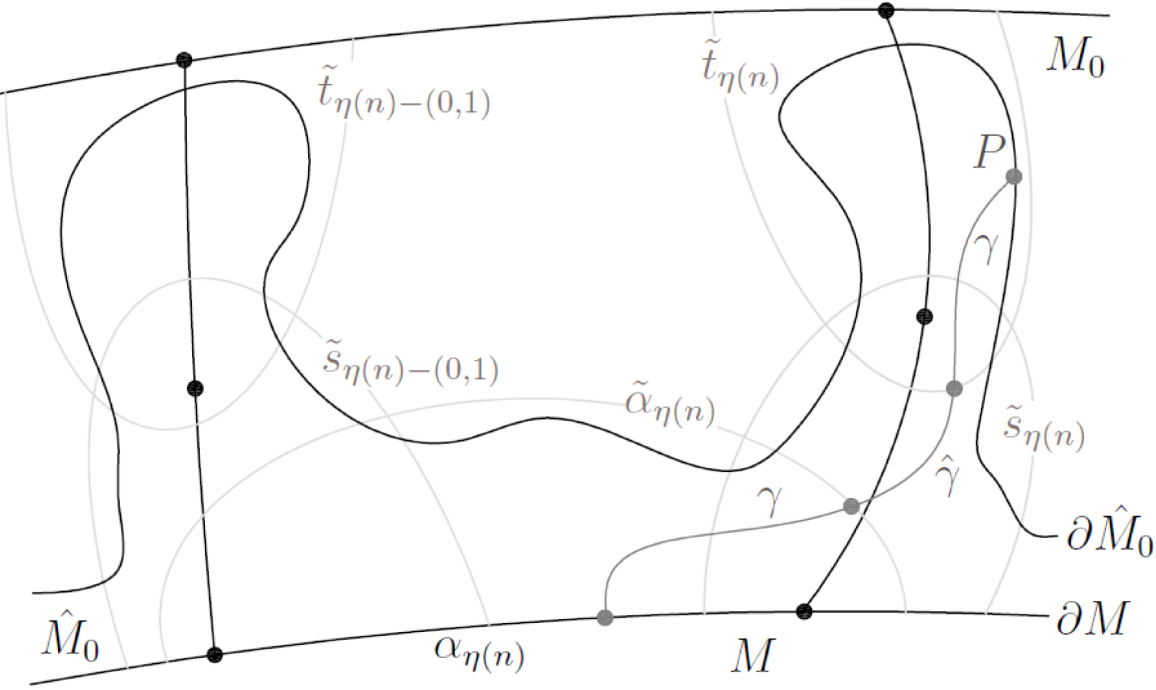}}
        \end{center}
        \vspace{-0.5cm}
\caption{$P\in \tilde{t}_{\eta(n)}.$}\label{fig:curvas1}
\end{figure}

\item Suppose now that $P\in \Omega_{\eta(n)}\cap(\tilde{s}_{\eta(n)-(0,1)} \cup \tilde{\alpha}_{\eta(n)}\cup\tilde{s}_{\eta(n)})$ (see Figure \ref{fig:curvas2-3}) for some  $n\in\{1,\ldots,km\}.$ By the preceding case, we can restrict ourselves to the case $\gamma\cap(\cup_{j=1}^{km}\tilde{t}_{\eta(j)})=\emptyset.$ By \eqref{eq:H1} again, $P \notin s_{\eta(n)-(0,1)}\cup\alpha_{\eta(n)}\cup s_{\eta(n)},$ and therefore there exists a sub-arc $\hat{\gamma}\subset\gamma$ contained in $\Omega_{\eta(n)}\cap (\tilde{s}_{\eta(n)-(0,1)}\cup \tilde{\alpha}_{\eta(n)}\cup\tilde{s}_{\eta(n)})$ connecting a point $Q\in s_{\eta(n)-(0,1)}\cup\alpha_{\eta(n)}\cup s_{\eta(n)}$ and $P.$    Since  (c1$_{km}$), $\hat{F}_0(Q)\in W_{\eta(n)},$ and (c2$_{km}$) and \eqref{eq:H2} give $\hat{F}_0(P)\in \partial \hat{\Dcal}\cap (W_{\eta(n)}+\doble{ e_{\eta(n)} }^\bot).$ Therefore \eqref{eq:comple1} implies that $\sqrt{d/\kappa(\Dcal)}+\epsilon_1<{\rm dist}(\hat{F}_0(P),\hat{F}_0(Q))\leq L(\hat{F}_0(\gamma)).$
\begin{figure}[ht]
    \begin{center}
   \scalebox{0.45}{\includegraphics{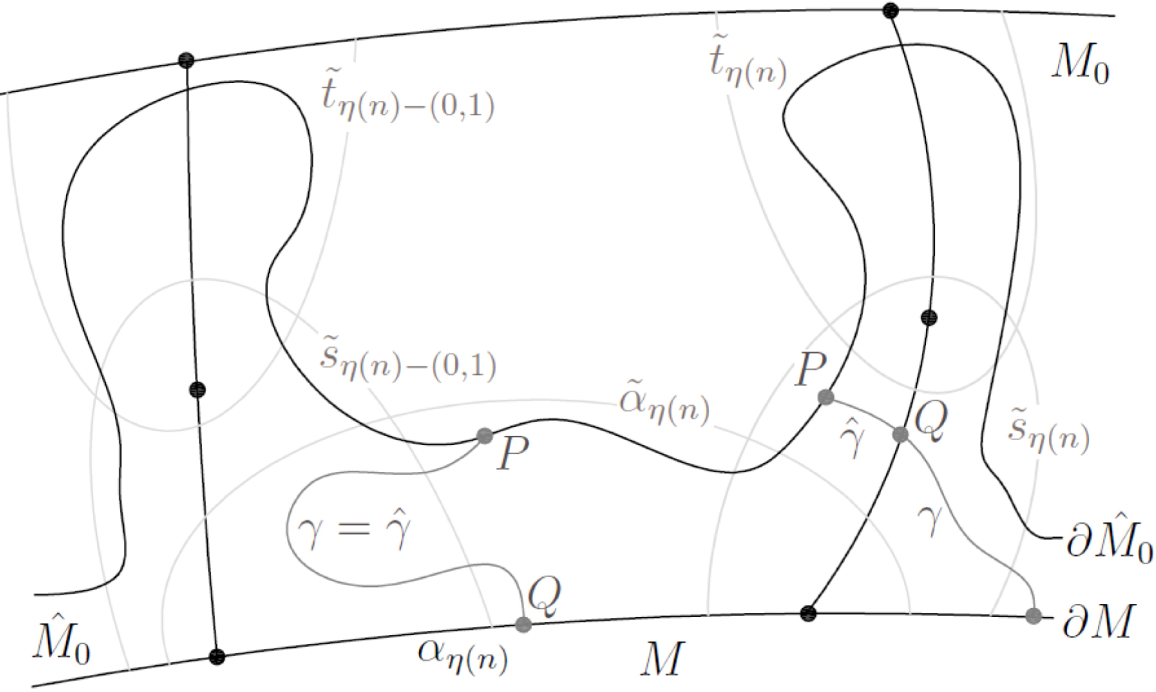}}
        \end{center}
        \vspace{-0.5cm}
\caption{$P\in \Omega_{\eta(n)}\cap(\tilde{s}_{\eta(n)-(0,1)} \cup \tilde{\alpha}_{\eta(n)}\cup\tilde{s}_{\eta(n)}).$}\label{fig:curvas2-3}
\end{figure}
\end{itemize}
\end{enumerate}

Finally define $\hat{F}:=\hat{F}_0|_{\hat{M}}\in {\sf N}(\hat{M}),$ where $\hat{M}\subset \hat{M}_0$ is close enough to $\hat{M}_0$ so that $\hat{F}(\partial\hat{M})\subset \hat{\Dcal}-\overline{\hat{\Dcal}_{-\epsilon_1}}$ and $\hat{F}$ satisfies (P1), (P2) and (P3). The region $\hat{M}$ and the null curve $\hat{F}$ solve Lemma \ref{lem:main}.


\end{document}